\documentclass[14pt]{amsart}
\usepackage{cases}
\usepackage{amsmath}
\usepackage{amsfonts}
\usepackage{bm}
\usepackage{color}
\usepackage{amssymb}
\usepackage{xypic}
\usepackage[all]{xy}
\usepackage{mathrsfs}
\usepackage{amsthm}
\usepackage[all]{xy}
\usepackage{graphics}
\usepackage{array}
\usepackage{graphicx} 
\usepackage{epstopdf}
\usepackage{float}
\usepackage{tikz}
\usepackage{tikz-cd}
\usetikzlibrary{arrows}
\usetikzlibrary{graphs}
\usepackage{tikz}
\usepackage{mathtools}
\usepackage{hyperref}
\usepackage{delarray}

\usepackage{subcaption}
\usepackage{booktabs}
\usepackage{extarrows}
\usepackage{rotating}

\newtheorem{thm}{Theorem}[section]
\newtheorem{lem}[thm]{Lemma}
\newtheorem{defi}[thm]{Definition}
\newtheorem{cor}[thm]{Corollary}
\newtheorem{prop}[thm]{Proposition}
\newtheorem{ex}[thm]{Example}
\newtheorem{rmk}[thm]{Remark}
\newtheorem{conj}[thm]{Conjecture}
\newtheorem{prob}[thm]{Problem}
\newtheorem{que}[thm]{Question}

\title
{Stability of Type A Mirkovi\'c-Vilonen Polytopes under Minkowski Sum via Weak Separation}

\author{Gleb A. Koshevoy \and Fang Li \and Lujun Zhang}

\address{Gleb A. Koshevoy
	\newline
	Institute for Information Transmission Problems,
	Russian Academy of Science of Moscow,
	Moscow 127051, Russia.}
\email{koshevoyga@gmail.com}

\address{Fang Li
\newline \footnote{The corresponding institution} School of Mathematical Sciences,
Zhejiang University,
Yuhangtang Road 866,
Hangzhou, Zhejiang 310058,
China P.R.}
\email{fangli@zju.edu.cn}

\address{Lujun Zhang\footnote{The corresponding author}
	\newline
	School of Mathematical Sciences,
	Zhejiang University,
	Yuhangtang Road 866,
	Hangzhou, Zhejiang 310058,
	China P.R. }
\email{12135007@zju.edu.cn}

\date{\today}

\newcommand{\lra}{\longrightarrow}

\newcommand{\ra}{\rightarrow}
\newcommand{\sdp}{\times\kern-.2em\vrule height1.1ex depth-.05ex}
\newcommand{\epi}{\lra \kern-.8em\ra}

\makeatletter
\@addtoreset{equation}{section}
\makeatother

 \normalbaselines
\begin{document}
\renewcommand{\thefootnote}{\alph{footnote}}

\renewcommand{\thefootnote}{\alph{footnote}}
\maketitle
\bigskip
	
\begin{abstract}
Mirkovi\'c--Vilonen (MV) polytopes play a key role in the representation theory of reductive algebraic groups, while the geometric behavior of prime MV polytopes under Minkowski addition remains a subtle open problem. This paper focuses on type A and regards Schubert matroid polytopes as fundamental prime MV building blocks. 

Using the crystal structure on MV polytopes, we strengthen Sanchez's compatibility condition and establish a necessary and sufficient condition: the positive Minkowski sum of such polytopes is again an MV polytope precisely when the indexing family is weakly separated. Working within discrete convex analysis, we relate discrete concave tropical Pl\"ucker functions to concave extensions on the hypercube and the resulting generalized matroid subdivisions, showing that weak separation is equivalent to the stability of these subdivisions under common refinement.

 We further clarify the intrinsic connection between our subdivision constructions and the hypersimplex matroid subdivisions developed by Early, providing a natural flag-type generalization of his classical results. We briefly discuss generalized positroids and generalized polypositroids, and identify the MV fan $\mathcal{MV}$ as the secondary fan of hypercube generalized positroid subdivisions. Accordingly, maximal weakly separated sets correspond to maximal cones in $\mathcal{MV}$ and produce the finest such subdivisions.
 
  This work unifies MV polytope theory with tropical matroid geometry, advances the understanding of compatibility phenomena in MV combinatorics, and offers new perspectives at the interface of representation theory and combinatorics.

\noindent\textbf{2020 Mathemathics Subject Classification: 05E45.}
\end{abstract}

\vspace{1cm}

\tableofcontents

\vspace{1cm}

\section{Introduction}
Mirkovi\'c-Vilonen \cite{MK} constructed subvarieties of the affine Grassmannian, known as \textbf{Mirkovi\'c-Vilonen cycles} (MV cycles for short). MV cycles are used to study the irreducible highest-weight modules of the Langlands dual group $G^{\vee}$ of a connected reductive complex algebraic group $G$.

To investigate the combinatorial properties of MV cycles, Anderson \cite{A} studied them via their moment polytope, which are called \textbf{Mirkovi\'c-Vilonen polytopes} (MV polytopes for short). In \cite{K2}, Kamnitzer extended the characterization of MV polytopes from type A to other Lie types using \textbf{generalized minors}. In particular, for type A, MV polytopes are parameterized by a class of functions $f:2^{[n]} \rightarrow \mathbb{Z}$ satisfying the following two conditions:
\begin{enumerate}
	\item (\textbf{Submodularity}) $f(X)+f(Y) \geq f(X \cap Y)+f(X \cup Y)$;
	\item (\textbf{Tropical Pl\"ucker relation}) $f(Xik)+f(Xj)=\max\{f(Xij)+f(Xk),f(Xjk)+f(Xi)\}$.
\end{enumerate}

where $X,Y \in 2^{[n]}$ and $i<j<k$ lie in $X^{c}$. Such functions are called \textbf{discrete concave tropical Pl\"ucker functions} (DCTP functions for short). This terminology arises from the fact that functions satisfying the type $\mathrm{A}$ conditions are essentially $M^{\natural}$-concave functions in discrete convex analysis. For general Lie types, the relations above are also referred to as \textbf{Berenstein--Zelevinsky data} (BZ data for short).

Kamnitzer explicitly discussed in Chapter 6 of \cite{K2} the structure of the parameter space of MV polytopes given by these BZ data, which is essentially a fan. By taking the quotient of this fan modulo its intrinsic maximal linearity subspace, we denote the resulting quotient fan by $\mathcal{MV}$. Kamnitzer referred to the polytopes corresponding to the generators (one-dimensional faces) in $\mathcal{MV}$ as \textbf{prime MV polytopes}. Anderson \cite{A} provided a complete characterization for low-rank cases such as $SL_3$, $SL_4$ and $SP_4$. In general, however, characterizing all prime MV polytopes and investigating their compatibility under the Minkowski sum remain open problems addressed in Chapter 6 of \cite{K2}.

For type $\mathrm{A}$ of arbitrary rank, all Schubert matroid polytopes turn out to be prime, as stated below.
\begin{prop}[Proposition \ref{10}]
	For any $I \in 2^{[n]} \setminus \{[1,j] \mid j \in [n]\}$, the Schubert matroid polytope $Q(\theta_I)$ is a prime MV polytope. Furthermore, if $I \neq J$, then $Q(\theta_I)$ and $Q(\theta_J)$ correspond to distinct stable MV polytopes.
\end{prop}

In the above proposition, $\theta_I$ denotes the rank function of the Schubert matroid and is also a DCTP function. The corresponding polytope $Q(\theta_I)$ is associated with a generator of the fan $\mathcal{MV}$. Sanchez \cite{S} (Theorem C) established a sufficient condition for the compatibility of these generators, that is, for a strongly separated set $\mathcal{C}$ in $2^{[n]}$, the positive Minkowski sum of the corresponding Schubert matroid polytopes
\[
Q_{\mathcal{C}}:=\sum\limits_{I \in \mathcal{C}}c_IQ(\theta_I),\quad c_I >0,
\]
is still an MV polytope.

Furthermore, utilizing the crystal structure on the set of all MV polytopes \cite{K}, we obtain a strengthened version Sanchez's theorem and present a necessary and sufficient condition.

\begin{thm}[Theorem \ref{minkow_schubert}]
	Let $\mathcal{C} \subseteq 2^{[n]}$. Then the positive Minkowski sum of Schubert matroid polytopes
	\[
	Q_{\mathcal{C}}:=\sum\limits_{I \in \mathcal{C}}c_IQ(\theta_I),\quad c_I>0,
	\]
	is an MV polytope if and only if $\mathcal{C}$ is a weakly separated set on $2^{[n]}$.
\end{thm}

The general notion of weak separation was introduced by Leclerc--Zelevinsky \cite{LZ}, which serves as the necessary and sufficient condition characterizing the quasi-commutativity of two quantum minors in the quantum coordinate ring of the flag variety.

On the other hand, we proceed purely from the perspective of discrete convex analysis without the representation-theoretic viewpoint. Each DCTP function $f$ can be regarded as a height function defined on the vertices of the hypercube $[0,1]^n$. Consider the concave extension $\widehat{f}$ of $f$ over the entire hypercube. The projection of the upper envelope of $\widehat{f}$ onto the hypercube induces a generalized matroid (g-matroid for short) subdivision of $[0,1]^n$ (see \cite{F} and \cite{FH}). In particular, we denote by $\Sigma_I$ the subdivision induced by $\theta_I$. We then obtain the following result:

\begin{thm}[Theorem \ref{g-positroid_ws}]
	\label{intro_g-positroid_ws}
	Let $\mathcal{C} \subseteq 2^{[n]}$. A common refinement of the g-matroid subdivisions $\{\Sigma_I : I \in \mathcal{C}\}$ of the hypercube $[0,1]^n$ remains a g-matroid subdivision if and only if $\mathcal{C}$ is a weakly separated set in $2^{[n]}$.
\end{thm}

From the perspective of polytope subdivisions, we reveal a profound connection between Theorem \ref{intro_g-positroid_ws} and the work of Early \cite{Ear}. Suppose $|I|=k$, and restrict the hypercube subdivision $\Sigma_I$ to the hypersimplex
\[
\Delta_{k,n}:=[0,1]^n \cap \Big\{x \in \mathbb{R}^n \,\Big|\, \sum_{i=1}^n x_i=k\Big\}.
\]
This restriction yields a matroid subdivision of the hypersimplex, which coincides exactly with the subdivision $\beta_I$ constructed by Early via translated blades (see \cite{Ear} and Remark \ref{rmk1} for details).

\begin{prop}[Proposition \ref{prop4}]
	For any $x \in \Delta_{k,n}$, we have
	\[
	\widehat{\theta}_I(x)=\min\big\{f_{\Pi_1}(x),f_{\Pi_2}(x),\dots,f_{\Pi_l}(x)\big\},
	\]
	where for $i \geq 2$,
	\[
	f_{\Pi_i}(x)=\sum_{m=1}^{E(S_{i-1})}x_m+\big|I_1\setminus [1,E(S_1)]\big|+|I_i|+|I_{i+1}|+\cdots+|I_l|.
	\]
	In particular,
	\[
	f_{\Pi_1}(x)=\sum_{m=1}^{E(S_{l})}x_m+\big|I_1\setminus [1,E(S_1)]\big|.
	\]
\end{prop}

The notations $S_1$, $I_i$, $E(S_1)$ and $\Pi_i$ are explained in Section \ref{SMABA}. Combined with the results of Fink et al.\cite{FMD}, we derive an explicit formula for $\widehat{\theta}_I$ on $[0,1]^n$ (see Theorem \ref{thm7}). Consequently, Theorem \ref{intro_g-positroid_ws} can be regarded as a flag analogue of Theorem 33 in Early \cite{Ear}.

In analogy with the work of Postnikov, Speyer, and Williams on positroid subdivisions and the positive tropical Grassmannian, we extend this notion of "positivity" from polymatroids to generalized polymatroids. 

Recall that a {\bf generalized polymatroid} (g-polymatroid for short)
$P$ is a polyhedron whose  edges 
parallel to vectors of the totally unmodular system $\mathbb{A}_n:=\{\pm e_i, \, e_i-e_j, \, i\neq j\in [n]\}$. This property is in essence equivalent to the exchange property for $g$-polymatroids (see, for example, \cite{M}).

The class of $g$-polymatroids is stable under Minkowski summation, taking faces, and  central symmetry $P\to -P$. Each of these operations preserves the property that edges have to be parallel to $\mathbb A_n$. Namely,  
edges of the sum $P+Q$ are parallel to edges either $P$ or $Q$; edges of any face remain edges of the polyhedron, central symmetry preserve directions from $\mathbb A_n$. $g$-polymatroids are not stable under intersections.

Two dimensional faces of a $g$-polymatroid are either 'hexagons'\footnote{We mean by hexagon also as a degenerate hexagon such as triangle, parallelogram or pentagon.} parallel to $g$-polymatroid of  a plane with basic vectors 
either $\{e_i, e_j\}$ or $\{e_i-e_j, e_j-e_k\}$, an $\mathbb A_2$-flat of $\mathbb A_n$; 
or a parallelogram parallel to a plane with  $\{e_i, e_j-e_k\}$ or $\{e_i-e_j, e_k\}$, $|\{i,j,k\}|=3$ or $\{e_i-e_l, e_j-e_k\}$ or $\{e_i-e_j, e_l-e_k\}$$|\{i,j,k,l\}|=4$. 

A class of {\bf generalized polypositroid} (g-polypositroid for short) is  a subclass of $g$-polymatroids, such that some of 2-faces parallel to  parallelograms  are forbidden.  Namely Proposition \ref{prop2}, for a $g$-polypositroid it is  not allow 
for $i<j<k$,  2-faces being parallel to planes with basic vectors   $\{e_i-e_k, e_j\}$  or , for  $i<j<k<l$, the 2-faces being parallel to planes with basic vectors $\{e_i-e_k, e_j-e_l\}$.

Let $\mathcal{GPP}(n)$ denote the set of all g-polypositroids in $\mathbb{R}^n$. Because of such a restriction, the class of $g$-polypositroids is not stability with respect to Minkowski summations and with respect to intersections. 

If a $g$-polypositroid is a 01-lattice polytope, it is called a \textbf{generalized positroid} (g-positroid for short). Several equivalent definitions of $g$-positroids are collected in Theorem \ref{thm2}.

Furthermore, using Theorem \ref{thm2}, we identify the fan $\mathcal{MV}$ with the secondary fan of the g-positroid subdivisions of the hypercube.

\vspace{0.2cm}
\begin{thm}[Theorem \ref{affine_DCTP}]
	For any DCTP function $f$, the affine domains of its concave extension over the hypercube belong to $\mathcal{GPP}(n)$, and the fan $\mathcal{MV}$ is exactly the secondary fan of the g-positroid subdivisions of the hypercube $[0,1]^n$.
\end{thm}

We obtain the following corollary from this theorem and our previous result in \cite{KLZ} that a refined arrangement of translated blades induced by a maximal weakly separated set yields a finest positroid subdivision.

\vspace{0.2cm}
\begin{cor}[Corollary \ref{finest_g_positroid}]
	Let $\mathcal{C}$ be a  maximal weakly separated set and $\theta:=\sum\limits_{I \in \mathcal{C}}c_I\theta_I$ ($c_I>0$).
	The g-positroid subdivision of $[0,1]^n$ induced by $\theta$ is finest.
	Furthermore, a maximal weakly separated family $\mathcal{C}$ in $2^{[n]}$ corresponds to a maximal cone of the fan $\mathcal{MV}$.
\end{cor}

\vspace{0.2cm}

This paper is organized as follows.

 In Section 2, we review type A Mirkovi\'c--Vilonen polytopes and discrete concave tropical Pl\"ucker (DCTP) functions, introduce Schubert matroids and blade arrangements, and prove that the positive Minkowski sum of Schubert matroid polytopes is an MV polytope if and only if the indexing collection is weakly separated (Theorem \ref{minkow_schubert}). We also relate weak separation to the stability of generalized matroid subdivisions of the hypercube. In Section 3, we define generalized polypositroids and generalized positroids, give their equivalent characterizations (Theorem \ref{thm2}), and show that the affine domains of concave extensions of DCTP functions are exactly g-positroids. In Section 4, we construct generalized positroids from generalized cyclic patterns and further discuss the structure of fan $\mathcal{MV}$.
 
 We use the following notations throughout the paper.  
 
 Let $m<n$ be positive integers. We write $[m,n]=\{m,m+1,\dots,n\}$, and in particular $[n]=\{1,2,\dots,n\}$. Denote by $\binom{[n]}{m}$ the collection of all $m$-element subsets of $[n]$. For a subset $I\subseteq[n]$, let $e_I=\sum_{i\in I}e_i$ be the sum of the corresponding unit vectors in $\mathbb{R}^n$, and let $x(I)=\sum_{i\in I}x_i$ denote the sum of thei corresponding variables.   If $X$ is a subset of $[n]$ and $i \in [n]\setminus X$, we abbreviate $X \cup \{i\}$ by $Xi$.

\vspace{1cm}

\section{Type A Mirkovi\'c-Vilonen polytopes and DCTP functions}
\label{GP}
\subsection{Preliminaries on generalized polymatroids and $M^{\natural}$-concave functions}\quad 

\begin{defi}
A (integer) polyhedron $Q \subseteq \mathbb{R}^N $ ($\mathbb{Z}^N$) is a (integer) generalized polymatroid if the directions of its edges belong to the totally unimodular set $ \mathbb{A}_N := \{ \pm e_i, e_i - e_j \} $.
\end{defi}

The original definition of g-polymatroid (short for generalized polymatroid) was introduced by Frank (1984), and the equivalence between the two was proven by Danilov and Koshevoy in \cite{DK}.

\begin{thm}[\cite{DK}]
\label{def_g_poly}
A polyhedron \( Q \) is a generalized polymatroid if and only if it is defined by the inequalities 
\[
g(A) \leq \sum_{i \in A} x_i \leq f(A), \quad A \subseteq N,
\]
where \( f \) is a submodular function, \( g \) is a supermodular function, and for any \( A \) and \( B \) there holds 
\[
f(B) - f(B - A) \geq g(A) - g(A - B).
\]
\end{thm}

\vspace{0.2cm}
A pair of functions \( g: 2^{N} \to \mathbb{R} \) (\( \mathbb{Z} \)) and \( f: 2^{N} \to \mathbb{R} \) (\( \mathbb{Z} \)) with \( g \leq f \) (pointwise) is called a \textbf{strong pair} if:
\begin{enumerate}
    \item \( f \) is submodular:
    \[
    f(A) + f(B) \geq f(A \cup B) + f(A \cap B)
    \]
    \item \( g \) is supermodular:
    \[
    g(A) + g(B) \leq g(A \cup B) + g(A \cap B)
    \]
    \item For any sets \( A, B \in 2^{[n]} \):
    \[
    f(B) - f(B - A) \geq g(A) - g(A - B)
    \]
\end{enumerate}
We always assume that $f(\emptyset)=g(\emptyset)=0$. Since such a g-polymatroid $Q$ is uniquely determined by strong pair $(f,g)$, we denote it by $Q=Q(f,g)$. In particular, by taking $g(A)=f^{\natural}(A):=f(N)-f(N-A)$, the g-polymatroid $Q(f^{\natural},f)$ is called a \textbf{base polyhedron}.

The class of g-polymatroids is stable under Minkowski summation, and for integer g-polymatroids, there holds 
\[
P(\mathbb{Z}) + Q(\mathbb{Z}) = (P + Q)(\mathbb{Z}),
\]
where \( P(\mathbb{Z}) \) denotes the set of integer points in \( P \).

In fact, there is a one-to-one correspondence between strong pairs and g-polymatroids.

\begin{prop}[\cite{FT}]
    \label{prop11}
    There holds $f(A)=\text{max}\{x(A)\;|\;x \in Q(f,g)\}$ and $g(B)=\text{min}\{x(B)\;|\;x \in Q(f,g)\}$. This implies that the strong pair $(f,g)$ of g-polymatroid $Q$ is unique.
\end{prop}

We call a g-polymatroid $Q$ in $\mathbb{R}^N$ \textbf{full-dimensional} if $Q$ is of dimension $|N|$,  while a base polyhedron $P$ in $\mathbb{R}^N$ is  \textbf{full-dimensional} if $P$ is of dimension $|N|-1$. Similar to the decomposition of a matroid into the direct sum of its connected components, a g-polymatroid  can also be decomposed based on its strong pairs

\begin{thm}[\cite{FKP}, Section 3]
\label{thm1}
	Let $Q=Q(f,g)$ is a g-polymatroid in $\mathbb{R}^N$, then there exists a unique partition $N=S_1 \sqcup \cdots \sqcup S_k \sqcup T$ and a decomposition 
    \[
    Q=P_1 \times\cdots \times P_k \times Q_0
    \]
    such that $P_i$ is a full-dimensional base polyhedron in $\mathbb{R}^{S_i}$ for $i=1,\cdots,k$ while $Q_0$ is a full-dimensional g-polymatroid in $\mathbb{R}^{T}$.
\end{thm}

Every $S_i$ is called a inclusion-minimal \textbf{sum-set} in \cite{FKP}, that is, the minimal subset of $N$ such that $f(S_i)=g(S_i)$. Besides, $S_1,\cdots,S_k$ and $T$ might be empty in this partition. 

\begin{cor}
\label{cor1}
    Let $Q$ be a g-polymatroid with the above decomposition, then $dim(Q)=n-k$.
\end{cor}

$M^{\natural}$-concave (convex) functions was originally introduced by Murota and Shioura by means of exchange axiom. Here, we adopt an equivalent definition that was mentioned by Fujishige and Hirai (see \cite{FH}, section 2.4).

\begin{defi}
Let $f: \mathbb{R}^N \rightarrow \mathbb{R} \cup \{-\infty\}$ be a polyhedral concave function such that 
\begin{enumerate}
    \item the effective domain $dom(f):=\{x \in \mathbb{R}^N\;|\;f(x) >-\infty \}$ is an integral g-polymatroid (base polyhedron),

    \item every linearity domain is an integral g-polymatroid (base polyhedron).
\end{enumerate}
Such a function $f$ is called a \textbf{$M^{\natural}$-concave function} (\textbf{M-concave function}), and $f$ can be identified with its restriction on the integral lattice $f:\mathbb{Z}^N \rightarrow \mathbb{R} \cup \{-\infty\}$ (equivalently, the range $\mathbb{R}$ can be chosen to be $\mathbb{Z}$).
\end{defi}

In particular, we consider $M^{\natural}$-concave functions whose effective domain is the hypercube $[0,1]^n$. In this case, the affine areas of this function are all g-polymatroids whose vertices are 0-1 lattice points. A 
g-polymatroid of this form is called a \textbf{generalized matroid} (abbreviated as \textbf{g-matroid}).

Then let us consider a tropical analogue of the algebraic 
Pl\"ucker relation on the minors of a matrix. 

\vspace{0.3cm}
\begin{defi}
	A \textbf{tropical Pl\"ucker function} is a function $f:2^{[n]} \rightarrow \mathbb{R}$ that satisfies the condition
	\begin{equation}
		\label{trop_plucker}
		f(Xik)+f(Xj)=\max\{f(Xij)+f(Xk),f(Xjk)+f(Xi)\}
	\end{equation}
	for any $i,j,k \in [n]$ with $i<j<k$ and $X \subseteq [n]\backslash\{i,j,k\}$.
\end{defi}

\vspace{0.3cm}
We can regard the tropical Pl\"ucker function $f$ as a function on the domain $\mathbb{Z}^n$ such that $f(e_X)=f(X)$ and $f(a)=-\infty$ for any $a \notin \{0,1\}^n$. We are interested in a class of tropical Pl\"ucker functions due to the next proposition, and we call them \textbf{discrete concave tropical Pl\"ucker} functions (\textbf{DCTP} functions for brief).

\begin{prop}
	\label{plucker_submodular}
	A tropical Pl\"ucker function f is $M^{\natural}$-concave if and only if $f$ is submodular.
\end{prop}

\begin{proof}
	Define a function $\widetilde{f}: \binom{[2n]}{n} \rightarrow \mathbb{R}$ by taking $\widetilde{f}(X):=f(X \cap [n])$ for every $X \in \binom{[2n]}{n}$. Then $f$ is $M^\natural$-concave if and only if $\widetilde{f}$ is $M$-concave (see \cite{M}, Chapter 6 ). By Proposition 2.2 of \cite{AS}, $\widetilde{f}$ is $M$-concave if and only if $f(Xi)+f(Xj) \geq f(X)+f(Xij)$ for any $X \in 2^{[n]}$ and $i,j  \notin X$. This condition is equivalent to the submodularity.
\end{proof}

\vspace{0.3cm}
One way to construct a DCTP function is through the laminar system (see \cite{DK} for example). A \textbf{laminar system} $\mathcal{L}$ is a collection of subsets such that for any $A, B \in \mathcal{L}$, either $A \subseteq B$ or $B \subseteq A$ or $A \cap B= \emptyset$.

\begin{prop}
	For a laminar system $\mathcal{L}$ that constitutes intervals, i.e. $\mathcal{L} \subseteq \mathcal{I}(n)$, assign a concave function $f_A :\mathbb{R} \rightarrow \mathbb{R}$ for any $A \in \mathcal{L}$. Then the restriction of function 
	\[
	F_{\mathcal{L}}(x):=\sum\limits_{A \in \mathcal{L}}f_A\left(\sum\limits_{i \in A}x_i\right) \;,x \in \mathbb{R}^n
	\]
	to $\{0,1\}^n$ is a DCTP function
\end{prop}

\begin{proof}
	Let us first check the submodularity of $F_{\mathcal{L}}$. For any $S,T \in 2^{[n]}$ and $A \in \mathcal{L}$, the relation
	\[
	f_A(|A \cap S|)+f_A(|A \cap T|) \geq f_A(|A \cap (S \cup T)|)+f_A(|A \cap (S \cap T)|)
	\]
	follows from the concavity of $f_A$. The tropical Pl\"ucker relation 
	\[
	f_A(|A \cap Xik|)+f_A(|A \cap Xj|)=\max\{f_A(|A \cap Xij|)+f_A(|A \cap Xik|), f_A(|A \cap Xjk|)+f_A(|A \cap Xi|)\}
	\]
	holds because $A$ is an interval. And since $\mathcal{L}$ is a laminar system, the sum of $f_A$ for all the $A \in \mathcal{L}$ satisfies the tropical Pl\"ucker relation.
\end{proof}

\vspace{0.3cm}
For any $\mathbb{Z}$-valued DCTP function $f$, we take it as a support function and define a polyhedron
\[
Q(f):=\{x \in \mathbb{R}^n\;|\;\sum\limits_{i \in S}x_i \leq f(S),\;\forall \;S\subsetneqq [n]\;\;\text{and} \;\; \sum\limits_{i=1}^nx_i=f([n]) \}
\]

The polyhedron defined in this way is also what Kamnitzer referred to as the type A (for the algebraic groups $GL_n$ or $SL_n$) \textbf{Mirkovi\'c}-\textbf{Vilonen polytope}, or simply \textbf{MV polytope}. In \cite{K2}, he established a one-to-one correspondence between stable MV polytopes and the Lusztig canonical basis. In later sections, we will further discuss the fan structure of all the collections of stable MV polytopes under Minkowski sum.

\vspace{1cm}

\subsection{Schubert matroid and blade arrangements}
\quad\\
\label{SMABA}
\indent Lattice path matroids were initially proposed in \cite{BMN} by Bonin et al, and it was recently proven by Sanchez (\cite{S}) that lattice path matroid polytopes precisely constitute the intersection class of MV polytopes and matroid polytopes. Among them, Schubert matroids, as a special class of lattice path matroids, have their rank functions computable through specific combinatorial techniques.\\
\indent Let $I=\{i_1,\cdots,i_k\;|\;1\leq i_1 \leq \cdots \leq i_k \leq n\}$ and $J=\{j_1,\cdots,j_k\;|\;1\leq j_1 \leq \cdots \leq j_k \leq n\}$ be two $k$-set in $[n]$, we denote by $I \leq_G J$ the condition that $i_1 \leq j_1, \cdots, i_k \leq j_k$.

\begin{defi}
	\label{lattice_path_matroid}
	Given two $I,J \in \binom{[n]}{k}$ with $I \leq_G J$, the \textbf{lattice path matroid} $M[I,J]$ is defined to the collection of $k$-sets as follows:
	\[
	M[I,J]:=\left\{L \in \binom{[n]}{k}\;\bigm| I \leq_G L\leq_G J\right\}.
	\]
	When $I=[k]$, the lattice path matroid $M[[k],J]$ is called \textbf{Schubert matroid}, denoted by $\Omega_J$.
\end{defi}

Let $M=M[I,J]$ be a lattice path matroid, its \textbf{rank function} $\theta_{I,J}(S):=\max\{|I|\mid I\subseteq S,I\in M\}$ for $S\subseteq [n]$, and the \textbf{matroid polytope} is defined to be
\[
Q(\theta_{I,J}):=Conv\{e_I\in\mathbb R^n\mid I\in M\}.
\]

In \cite{MTY}, a possible method for calculating the rank function of Schubert matroid was proposed and later proved in [\cite{FMD}, Theorem 14]. The detailed construction is as follows.\\
\indent To facilitate left-to-right reading of the symbols in the boxes, we transpose the $n \times 1$ grids from \cite{FMD} into the $1\times n$ grids and number them sequentially from left to right as $1,2,\cdots,n$. First, for a given $k$-set $I$, we shade the boxes whose indices belong to $I$. Then for any set $J \in 2^{[n]}$, we fill each box $i$ with a symbol according to the following rules: 

\begin{itemize}
	\item fill "(" if $i \in J$ and $i \notin I$;
	\item fill ")" if $i \notin J$ and $i \in I$;
	\item fill "$*$" if $i \in I \cap J$;
	\item do not fill any symbol if $i \notin I\cup J$.
\end{itemize}
The collections of all the symbols from left to right form a word denoted by $W_I(J)$. Let
\[
\theta_I(J):=\#\left\{{\rm paired}\;()'s\;{\rm in\;word}\;W_I(J)\right\}+\#\left\{*'s \;{\rm in\;word}\;W_I(J)\right\}.
\]
with pairing performed in the standard 'inside-out' manner. Such a function $\theta_I: 2^{[n]} \rightarrow \mathbb{N}$ is exactly the rank function of Schubert matroid $\Omega_I$.\\

\begin{ex}
	As shown in Figure \ref{fig1}, suppose that $n=10$, $I=\{4,5,6,8,9\}$ and $J=\{1,2,3,4,7,8\}$,
	\begin{figure}[H]
		\centering
		\includegraphics[width=0.5\linewidth]{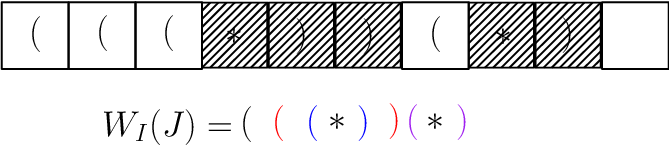}
		\caption{}
		\label{fig1}
	\end{figure}
	The half-brackets of the same color in the word $W_I(J)$ are paired into ()'s. Then $\theta_I(J)=5$.
\end{ex}
When $|J|=k$, the rank function $\theta_I(J)$ admits an explicit formula, expressed as the pointwise minimum of finitely many linear functions. Let us first define the affine area of this piecewise linear function. Suppose that $I=I_1\sqcup \cdots \sqcup I_l$ be the unique decomposition of $k$-set $I$ into cyclic intervals and let $C_i$ be the cyclic interval bridging $I_{i-1}$ and $I_i$. Then $[n]=(C_1,I_1,\cdots,C_l,I_l)$ is the decomposition of $[n]$ into a collection of adjacent cyclic intervals. For simplicity, we denote $S_i:=(C_i,I_i)$ and assume that $1 \in S_1$. We call $l$ the \textbf{partition number} of $I$ and denote it by $\sigma(I)=l$.

\begin{defi}
	\label{def1}
	Let $\Pi_i=\left[ \left(S_i\right)_{|I_i|}, \left(S_{i+1}\right)_{|I_{i+1}|},\cdots,\left(S_{i-1}\right)_{|I_{i-1}|}\right]$ $(i=1,2,\cdots,l)$ be the region in $\mathbb{R}^n$ defined by the following (in)equalities:
	\[
	\begin{split}
		&x(S_i) \geq |I_i|;\\
		&x(S_i\cup S_{i+1}) \geq |I_i|+|I_{i+1}|;\\
		&\vdots\\
		&x(S_i\cup S_{i+1}\cup \cdots \cup S_{i-2}) \geq |I_i|+|I_{i+1}|+\cdots |I_{i-2}|;\\
		&x(S_i\cup S_{i+1}\cup \cdots \cup S_{i-1})=x([n])=|I_i|+|I_{i+1}|+\cdots |I_{i-1}|=k.
	\end{split}
	\]
	Here, the subscripts of $S_i,I_i$ and $\Pi_i$ are taken modulo $l$. We maintain this convention throughout the subsequent discussion.
	
\end{defi}

In \cite{H}, the author investigated a kind of subdivisions of a given point configuration called \textbf{k-splits}. The point configuration refers to vertices of the given polytope in our context. A k-split of a polytope $\mathcal{P}$ is a coarsest subdivision $\Sigma$ with $k$ maximal face such that $\Sigma$ has a common face of codimension $k-1$ for all the maximal faces. We denote this unique common face of $\Sigma$ by $\mathcal{F}_{\Sigma}^{\mathcal{P}}$, if $\mathcal{P}$ is already clear from the context, we simply denote it by $\mathcal{F}_{\Sigma}$. When we do not want to emphasize the specific number $k$, we may use the term \textbf{multi-split} to refer to this subdivision. Besides, the author proved the following result.

\begin{thm}[\cite{H},Theorem 4.9]
	\label{thm6}
	All the k-splits are coarsest regular subdivisions. 
\end{thm}

Then, Early studied a class of geometric objects called permutohedral blades and applied them to the muti-split of the hypersimplex. In \cite{Ear2} and \cite{Ear}, he proved the following result.

\begin{prop} ( \cite{Ear2},  \cite{Ear})
	\label{prop7}
	Polyhedral cones $\Pi_i$ $(i=1,2,\cdots,l)$ form all the maximal cones of a complete polyhedral fan in the hyperplane $\{x \in \mathbb{R}^n\;|\;\sum\limits_{i=1}^n x_i=k\}$. Meanwhile, the intersections $\Pi_i \cap \Delta_{k,n}$ $(i=1,2,\cdots,l)$ form a $l$-split of $\Delta_{k,n}$.
\end{prop}

Subsequently, we discovered that Early's work on hypersimplex multi-splits is fundamentally linked to the rank function of Schubert matroids, which can be formulated as follows.

\begin{lem}
	\label{lem1}
	Suppose that $I$ is not a cyclic interval i.e. $l \geq 2$. Then for any $J \in \binom{[n]}{k}$ and $e_J \in \Pi_i \;(i \geq 2)$, we have 
	\[
	\theta_I(J)=|J \cap [1,E(S_{i-1})]|+|I_1\backslash [1,E(S_1)]|+|I_{i}|+|I_{i+1}|+\cdots |I_l|
	\]
	where $E(S_i)$ represents the last element of cyclic interval $S_i$. Particularly, for $e_J \in \Pi_1$, $\theta_I(J)=|J \cap [1,E(S_{l})]|+|I_1\backslash [1,E(S_1)]|$. 
\end{lem}

\begin{proof}
	\begin{figure}[H]
		\centering
		\includegraphics[width=12cm]{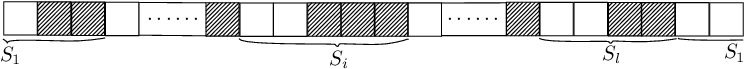}
		\caption{}
		\label{fig2}
	\end{figure}
	As shown in the Figure \ref{fig2}, we count the contributions of paired ()'s and $*$'s to $\theta_I(J)$ when $e_J$ is in the different region $\Pi_i$. We first address the case that $e_J \in \Pi_i$ for $i \geq 2$, then all the )'s in the shadowed boxes of $(S_i,S_{i+1},\cdots,S_l)$ are paired into ()'s by those ('s within the interval. Thus, the total number of paired ()'s and $*$'s in the interval $(S_i,S_{i+1},\cdots,S_l)$ is $|I_{i}|+|I_{i+1}|+\cdots |I_l|$. On the other hand, by equality $\sum\limits_{i}x_i=k$, we obtained that 
	\[
	\begin{split}
		&e_J(S_{i-1}) \leq |I_{i-1}|;\\
		&e_J(S_{i-2} \cup S_{i-1}) \leq |I_{i-2}|+|I_{i-1}|;\\
		&\vdots\\
		&e_J(S_2 \cup \cdots \cup S_{i-1}) \leq |I_{2}|+\cdots+|I_{i-1}|.
	\end{split}
	\]
	Thus the ('s in the interval $(S_2,\cdots,S_{i-1})$ are all paired by )'s within the interval. So the total number of ()'s and $*$'s in $(S_2,\cdots,S_{i-1})$ is $|J \cap (S_2\cup \cdots S_{i-1})|$ (Note that when $i=2$, this part need not be counted). It remains only to consider the contributions within $S_1$. As shown in Figure \ref{fig34}, we discuss two cases based on whether $n$ is within the cyclic interval $I_1$.
	
	\begin{figure}[htbp]
		\centering
		\begin{subfigure}[b]{0.48\textwidth}
			\includegraphics[width=6.5cm]{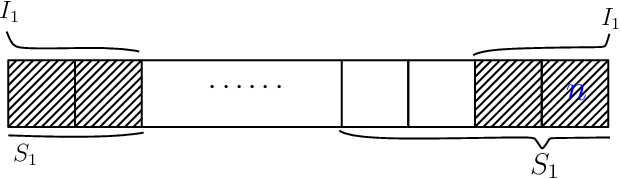}
			\caption{Case 1: $n \in I_1$}
			\label{subfig3}
		\end{subfigure}
		\hfill 
		\begin{subfigure}[b]{0.48\textwidth}
			\includegraphics[width=6cm]{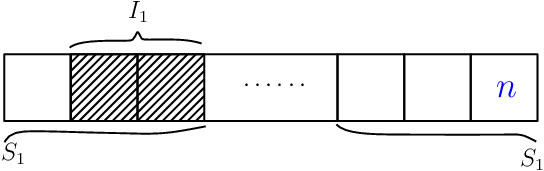}
			\caption{Case 2: $n \notin I_1$}
			\label{subfig4}
		\end{subfigure}
		\caption{}
		\label{fig34}
	\end{figure}
	In case 1, Since $I_1\backslash [1,E(S_1)] \neq \emptyset$ and $e_J(S_i \cup \cdots \cup S_l \cup S_1) \geq |I_i|+\cdots+|I_l|+|I_1|$, the )'s in $I_1\backslash [1,E(S_1)]$  are all paired by ('s from the interval $(S_i,\cdots,S_l,C_1)$. Within the interval $[1,E(S_1)]$, all contributions come from $*$'s. Thus the total contribution is $|J \cap [1, E(S_1)]|+|I_1\backslash [1,E(S_1)]|$.
	
	In case 2, since $e_J(S_1 \cup \cdots \cup S_{i-1}) \leq |I_1|+\cdots+|I_{i-1}|$, the ('s in the interval $[1,E(S_1)]$ are all paired into ()'s by )'s from interval $(I_1,S_2,\cdots,S_{i-1})$. In this case, $|I_1 \backslash [1,E(S_1)]|=0$, but we choose to keep it in the equality. The total contribution is also $|J \cap [1, E(S_1)]|+|I_1\backslash [1,E(S_1)]|$.
	
	Summing the contributions from all parts:
	\[
	\begin{split}
		\theta_I(J) &=|I_{i}|+|I_{i+1}|+\cdots |I_l|+|J \cap (S_2\cup \cdots S_{i-1})|+|J \cap [1, E(S_1)]|+|I_1\backslash [1,E(S_1)]|\\
		&=|J \cap [1,E(S_{i-1})]|+|I_1\backslash [1,E(S_1)]|+|I_{i}|+|I_{i+1}|+\cdots |I_l|.
	\end{split}
	\]
	In particular, the case that $e_J \in \Pi_1$ can be vertified separately.
\end{proof}

Since the Schubert matroid polytope $P(\Omega_I)$ is a MV-polytope, the rank function $\theta_I$ is a submodular TP function. Let $\widehat{\theta}_I:[0,1]^n \rightarrow \mathbb{R}$ denote its concave extension. Inspired by Lemma \ref{lem1}, We obtain an explicit formula for $\widehat{\theta}_I(x)$ when $x \in \Delta_{k,n}$.

\begin{prop}
	\label{prop4}
	For any $x \in \Delta_{k,n}$, $\widehat{\theta}_I(x)={\rm min}\{f_{\Pi_1}(x),f_{\Pi_2}(x), \cdots,f_{\Pi_l}(x)\}$ where
	\[
	f_{\Pi_i}(x)=\sum\limits_{m=1}^{E(S_{i-1})}x_m+|I_1\backslash [1,E(S_1)]|+|I_{i}|+|I_{i+1}|+\cdots |I_l|
	\]
	for $i \geq 2$. Particularly, $f_{\Pi_1}(x)=\sum\limits_{m=1}^{E(S_{l})}x_m+|I_1\backslash [1,E(S_1)]|$.
\end{prop}

\begin{proof}
	The point is that the minimum of linear functions yields a concave piecewise linear function. Fix $i$, and for any $j > i$, $f_{\Pi_j}(x) \geq f_{\Pi_i}(x)$ if and only if
	\[
	\begin{split}
		f_{\Pi_j}(x)-f_{\Pi_i}(x) &=\left(\sum\limits_{m=1}^{E(S_{j-1})}x_m-\sum\limits_{m=1}^{E(S_{i-1})}x_m\right)+(|I_{j}|+|I_{j+1}|+\cdots |I_l|)-(|I_{i}|+|I_{i+1}|+\cdots |I_l|)\\
		&=x(S_{i} \cup \cdots S_{j-1})-(|I_i|+\cdots+|I_{j-1}|) \geq 0.
	\end{split}
	\]
	By analogous computations, for $j < i$, $f_{\Pi_j}(x) \geq f_{\Pi_i}(x)$ is equivalent to that $x(S_{j} \cup \cdots S_{i-1}) \leq |I_j|+\cdots+|I_{i-1}|$. This is equivalent to 
	\[
	x(S_i \cup \cdots \cup S_l\cup \cdots \cup S_{j-1}) \geq |I_i|+\cdots+|I_l|+\cdots +|I_{j-1}|.
	\]
	by assumption that $x \in \Delta_{k,n}$. 
	
	Therefore,  
	$\widehat{\theta}_I(x)=f_{\Pi_i}(x)$ if and only if $x \in \Pi_i=\left[ \left(S_i\right)_{|I_i|}, \left(S_{i+1}\right)_{|I_{i+1}|},\cdots,\left(S_{i-1}\right)_{|I_{i-1}|}\right]$. And $\widehat{\theta}_I(e_J)=f_{\Pi_i}(e_J)$ is directly by Lemma \ref{lem1} when $e_J \in \Pi_i$
\end{proof}

\vspace{0.2cm}
\begin{rmk}
	\label{rmk1}
	Let $((S_1)_{|I_1|},(S_2)_{|I_2|},\cdots,(S_l)_{|I_l|})$ be the union of all the codimension 1 faces of cones $\Pi_i$ $(i=1,\cdots,l$. This is the geometric object call \textbf{blade} by Early \cite{Ear}. If $S_i=\{i\}$ and $I_i=\emptyset$, we use $\beta:=(1,2,\cdots,n)$ to denote the \textbf{standard blade} and use $\beta_I:=(1,2,\cdots,n)_{e_I}$ to denote the \textbf{translated blade} which is obtained by moving the standard blade $\beta$ from the origin to vertex $e_I$. Proposition \ref{prop4} provide another way to prove that blades are tropical hypersurfaces by the concave extension of rank function of Schubert matroid.
\end{rmk}

\vspace{0.2cm}
\begin{thm}[\cite{Ear}]
	\label{thm3}
	Let $e_I$ be a vertex of $\Delta_{k,n}$, the translated blade $\beta_{I}=(1,2,\cdots,n)_{e_{I}}$ induces a $l$-split of $\Delta_{k,n}$ such that 
	\[
	(1,2,\cdots,n)_{e_{I}} \cap \Delta_{k,n}=((S_1)_{|I_1|},(S_2)_{|I_2|},\cdots,(S_l)_{|I_l|}) \cap \Delta_{k,n}
	\]
\end{thm}

Combining Proposition \ref{prop4} and Theorem \ref{thm3}, we obtain the follwing result

\begin{prop}
	\label{prop9}
	For any $I \in \binom{[n]}{k}$, the affine areas of $\widehat{\theta}_I$ induce a g-matroid subdivision of $[0,1]^n$. By restricting this subdivision at $\sum\limits_{i=1}^n x_i=k$, the truncated matroid subdivision coincides precisely with that induced by translated blade $\beta_I$.
\end{prop}

\vspace{0.2cm}
\begin{ex}
	When $n=6$ and $I=\{1,3,5\}$, $\widehat{\theta}_I(x)={\rm min}\{f_{\Pi_1}(x),f_{\Pi_2}(x),f_{\Pi_3}(x)\}$ where $x \in \Delta_{3,6}$ and 
	\[
	\begin{split}
		&f_{\Pi_1}(x)=x_1+x_2+x_3+x_4+x_5;\\
		&f_{\Pi_2}(x)=x_1+2;\\
		&f_{\Pi_3}(x)=x_1+x_2+x_3+1.
	\end{split}
	\]
	Meanwhile, the positroid subdivision induced by $\beta_I$ is a 3-split of $\Delta_{3,6}$.
\end{ex}

In Proposition \ref{prop4}, we only provided the explicit formula for the rank function of Schubert matroid $\Omega_I$ on all $|I|$-element subsets. However, in discrete convex analysis, we can achieve an equivalent transformation between $M^{\natural}$-concave function and $M$-concave function through a homogenization-like operation. Next, we will utilize this tool to derive the explicit formula for the rank function of Schubert matroid on the whole domain $2^{[n]}$. 

Let $\pi: \mathbb{R}^{2n} \rightarrow \mathbb{R}^{n}$ be the projection map
\[
\pi : (x_1,x_2\cdots,x_n,x_{n+1},\cdots ,x_{2n}) \rightarrow  (x_1,\cdots,x_n).
\]
One can check that $\pi(\Delta_{n,2n})=[0,1]^n$. And one can lift any $M^{\natural}$-concave funtion $f$ on $2^{[n]}$ to a $M$-concave funtion $\tilde{f}$ on $\binom{[2n]}{n}$ by letting $\tilde{f}(X):=f(X \cap [n])$ for any $X \in \binom{[2n]}{n}$. 

The following definition is from Oh et al.
\begin{defi}[\cite{OPS}]
	Let $I \in \binom{[n]}{k}$, define $ \mathbf{pad}(I)$:=$I \cup \{n+k+1,\;n+k+2,\cdots,2n\}$ to be an element in $ \binom{[2n]}{n}$. 
\end{defi}

Then we have a Schubert matroid $\Omega_{\mathbf{pad}(I)}:=\{L \in \binom{[2n]}{n}\;|\;L \leq_G \mathbf{pad}(I)\}$ on the ground set $[2n]$. And we denote the rank function of $\Omega_{\mathbf{pad}(I)}$ by $\theta_{\mathbf{pad}(I)}$. There is a crucial relation between $\theta_{\mathbf{pad}(I)}$ and $\theta_I$ as shown in the following lemma.

\begin{lem}
	\label{lem5}
	For any $I \in \binom{[n]}{k}$ and $X \in \binom{[2n]}{n}$, we have
	\[
	\theta_{\mathbf{pad}(I)}(X)=\theta_I(X \cap [n])+n-k
	\]
\end{lem}
\begin{proof}
	To describe the relation between $\theta_{\mathbf{pad}(I)}(X)$ and $\theta_I(X \cap [n])$, we only need to extend the $1 \times n$ grid $\mathcal{D}_I$ to the $1 \times 2n$ grid $\mathcal{D}_{\mathbf{pad}(I)}$ as Figure \ref{fig25}.
	
	\begin{figure}[H]
		\centering
		\includegraphics[width=10cm]{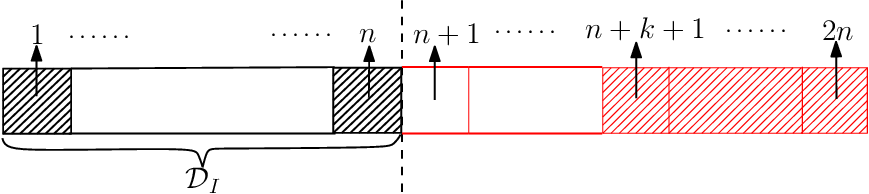}
		\caption{The picture of $\mathcal{D}_{\mathbf{pad}(I)}$ when $1 \in I$ and $n \in I$}
		\label{fig25}
	\end{figure}
	When $|X \cap [n]| \leq k$, then $|X \cap [n+1,2n]| \geq n-k$. Thus the right brackets )'s in  shaded boxes labeled by $n+k+1,\cdots ,2n$ are all paired and we get the above relation.
	
	When $|X \cap [n]| < k$, the boxes  labeled by $1,\cdots,n$ provide at least $|X \cap [n]|-k$ unpaired left brackets ('s. The total number of these unpaired ('s and the intersection of $X$ with $[n+1,2n]$ is at least $|X \cap [n]|-k+|X \cap [n+1,2n]|=|X \cap [n]|-k+n-|X \cap [n]|=n-k$. Therefore, the right brackets )'s in  shaded boxes labeled by $n+k+1,\cdots ,2n$ are all paired, then we get the same relation.
\end{proof}

\vspace{0.2cm}
From Lemma \ref{lem5}, we know that the value of $\theta_I$ on $2^{[n]}$ can be viewed as a projection of $\theta_{\mathbf{pad}(I)}$'s value on $\binom{[2n]}{n}$ with an additional translation. Fortunately, Proposition \ref{prop4} provides a specific piecewise linear expression for $\theta_{\mathbf{pad}(I)}$ on the domain $\binom{[2n]}{n}$. To get the explicit piecewise linear formula of $\theta_I$ on the whole domain $2^{[n]}$, we need to describe the image of affine areas of $\theta_{\mathbf{pad}(I)}$ under the projection $\pi$.  We use $\Sigma_I$ to denote the generalized positroid subdivision of hypercube $[0,1]^n$ induced by $\theta_I$.
\vspace{0.2cm}

\begin{prop}
	\label{prop8}
	Let $\mathbf{pad}(I)=\widetilde{I_1} \sqcup \cdots \widetilde{I}_{\sigma(\mathbf{pad}(I))}$ be the cyclic interval decomposition of $\mathbf{pad}(I)$ over ground set $[2n]$. Simlarly, for $\mathbf{pad}(I)$, we analogouly define $\widetilde{C}_i$ and $\widetilde{S}_i$ (as $C_i$, $S_i$ of $I$ over ground set $[n]$ for $i=1,\cdots,\sigma(I)$) over ground set $[2n]$ for $i=1,\cdots, \sigma(\mathbf{pad}(I))$. Then the partition number $\sigma(\mathbf{pad}(I))=\sigma(I)$ or $\sigma(I)+1$.
	\begin{enumerate}
		\item When $\sigma(\mathbf{pad}(I))=\sigma(I)=l$, this is equivalent to that $1 \in I$ and $n \notin I$ as Figure \ref{fig26}. In this case, $\widetilde{S}_1=[1,E(S_1)] \sqcup [E(S_l)+1,2n]$, $\widetilde{I}_1=[1,E(S_1)] \sqcup [n+k+1,2n]$ while $\widetilde{S}_i=S_i$, $\widetilde{I}_i=I_i$ for $i=2,\cdots,l$. Then $\Sigma_I$ is a $l$-split of $[0,1]^n$ and the common face of $\Sigma_I$ is
		\[
		\mathcal{F}_{\Sigma_I}=\Delta_{|I_2|,S_2} \times \Delta_{|I_3|,S_3} \times \cdots \Delta_{|I_l|,S_l} \times [0,1]^{S_1},
		\]
		where $\Delta_{|I_i|,S_i}:=\{x \in [0,1]^{S_i}\;|\;\sum\limits_{j \in S_i} x_i=|I_i|\}$.
		\item When $\sigma(\mathbf{pad}(I))=\sigma(I)=l+1$, $\Sigma_I$ induce a $(l+1)$-split of $[0,1]^n$, and there are three cases:\\
		
		\begin{enumerate}
			\item If $1 \in I$ and $n \in I$ as Figure \ref{fig25'}, then $\widetilde{S}_1=[1,E(S_1)] \sqcup [n+1,2n]$,  $\widetilde{I}_1=[1,E(S_1)] \sqcup [n+k+1,2n]$, $\widetilde{S}_{l+1}=[E(S_l)+1,n]$, $\widetilde{I}_{l+1}=I_1 \cap \widetilde{S}_{l+1}$ while $\widetilde{S}_i=S_i$, $\widetilde{I}_i=I_i$ for $i=2,\cdots,l$. The common face
			\[
			\mathcal{F}_{\Sigma_I}=\Delta_{|I_2|,S_2} \times \Delta_{|I_3|,S_3} \times \cdots \Delta_{|I_l|,S_l} \times \Delta_{|\widetilde{I}_{l+1}|, \widetilde{S}_{l+1}} \times [0,1]^{[1,E(S_1)]}.
			\]
			\item If $1 \notin I$ and $n \in I$ as Figure \ref{fig27}, then $\widetilde{S}_{l+1}=[n+1,2n]$, $\widetilde{I}_{l+1}=[n+k+1,2n]$ while $\widetilde{S}_i=S_i$, $\widetilde{I}_i=I_i$ for $i=1, \cdots, l$. The common face
			\[
			\mathcal{F}_{\Sigma_I}=\Delta_{|I_1|,S_1} \times \Delta_{|I_2|,S_2} \times \cdots \Delta_{|I_l|,S_l}.
			\]
			\item If $1 \notin I$ and $n \notin I$ as Figure \ref{fig28}, then $\widetilde{S}_1=[1,E(S_1)]$, $\widetilde{I}_1=I_1$, $\widetilde{S}_{l+1}=[E(S_l)+1,2n]$, $\widetilde{I}_{l+1}=[n+k+1,2n]$ while $\widetilde{S}_i=S_i$, $\widetilde{I}_i=I_i$ for $i=2, \cdots, l$. The common face
			\[
			\mathcal{F}_{\Sigma_I}=\Delta_{|\widetilde{I}_1|, \widetilde{S}_1} \times \Delta_{|I_2|,S_2} \times \Delta_{|I_3|,S_3} \times \cdots \Delta_{|I_l|,S_l} \times [0,1]^{[E(S_l)+1,n]}.
			\]
		\end{enumerate}
	\end{enumerate}
\end{prop}

\begin{figure}[htbp]
	\centering
	\begin{subfigure}[t]{0.3\textwidth}
		\includegraphics[width=6.5cm]{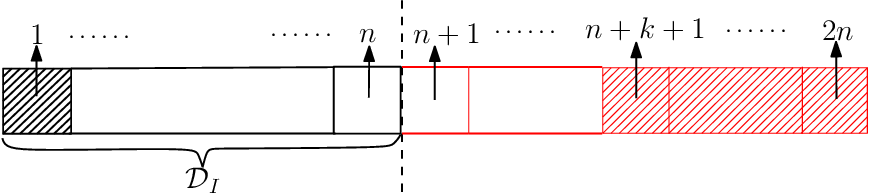}
		\caption{$1 \in I$ and $n \notin I$}
		\label{fig26}
	\end{subfigure}
	\hspace{0.15\textwidth}
	\begin{subfigure}[t]{0.3\textwidth}
		\includegraphics[width=6.5cm]{fig25.eps}
		\caption{$1 \in I$ and $n \in I$}
		\label{fig25'}
	\end{subfigure}\\
	\vspace{0.3cm}
	\begin{subfigure}[t]{0.3\textwidth}
		\includegraphics[width=6.5cm]{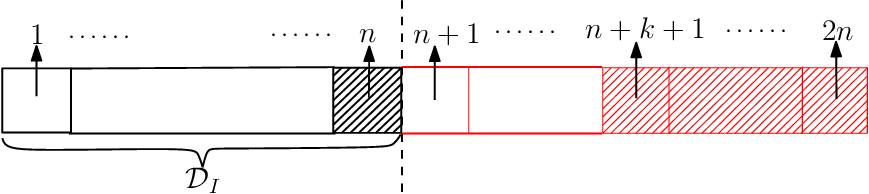}
		\caption{$1 \notin I$ and $n \in I$}
		\label{fig27}
	\end{subfigure}
	\hspace{0.15\textwidth}
	\begin{subfigure}[t]{0.3\textwidth}
		\includegraphics[width=6.5cm]{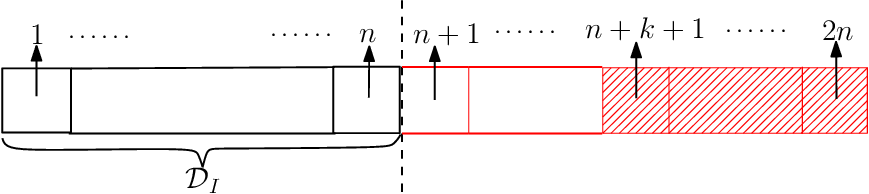}
		\caption{$1 \notin I$ and $n \notin I$}
		\label{fig28}
	\end{subfigure}
	\caption{}
	\label{fig25'28}
\end{figure}

\vspace{0.2cm}
\begin{proof}
	Since $\mathbf{pad}(I)$ is obtained from $I$ by adding an interval $[n+k+1, 2n]$, the number of cyclic intervals in the decomposition of $\mathbf{pad}(I)$ either increases by 1 or remains unchanged.
	We provide a proof for (1) , since the remaining cases follow by analogous arguments. 
	
	In case (1), when $\sigma(\mathbf{pad}(I))=\sigma(I)$, i.e.,$[n+k+1,2n]$ is absorbed by a cyclic interval of $I$, i.e., $1 \in I$ and $n \notin I$. By the assumption that $1 \in \widetilde{S}_1$, we obtain that $\widetilde{S}_1=[1,E(S_1)] \sqcup [E(S_l)+1,2n]$, $\widetilde{I}_1=[1,E(S_1)] \sqcup [n+k+1,2n]$ while $\widetilde{S}_i=S_i$, $\widetilde{I}_i=I_i$ for $i=2,\cdots,l$.
	
	Let $\beta_{\mathbf{pad}(I)}$ be the translated blade subdivision of $\Delta_{n,2n}$ induced by $\mathbf{pad}(I)$. Thus, $\beta_{\mathbf{pad}(I)}$ is a $l$-split of $\Delta_{n,2n}$. The image of the full-dimensional positroid $\pi(\widetilde{\Pi}_i \cap \Delta_{n,2n})=\pi(\widetilde{\Pi}_i) \cap [0,1]^n$ under projection $\pi$ is a full-dimensional generalized positroid. The common face of $\beta_{\mathbf{pad}(I)}$ is
	\[
	\mathcal{F}_{\beta_{\mathbf{pad}(I)}}=\bigcap\limits_{i=1}^l\left\{x \in \Delta_{n,2n}\;|\;x_{\widetilde{S}_i}=|\widetilde{I_i}|\}\right\}=\Delta_{|\widetilde{I}_{1}|, \widetilde{S}_{1} } \times \cdots \Delta_{|\widetilde{I}_{l}|, \widetilde{S}_{l}}.
	\]
	Consider the image of every connect component of positroid $\mathcal{F}_{\beta_{\mathbf{pad}(I)}}$ under $\pi$, we have $\pi(\Delta_{|\widetilde{I}_{i}|, \widetilde{S}_{i} })=\Delta_{|\widetilde{I}_{i}|, \widetilde{S}_{i} }=\Delta_{|I_i|,S_i}$ for $i=2,\cdots,l$ and $\pi(\Delta_{|\widetilde{I}_{1}|, \widetilde{S}_{1}})=[0,1]^{S_1}$. Thus,
	\[
	\mathcal{F}_{\Sigma_I}=\pi(\mathcal{F}_{\beta_{\mathbf{pad}(I)}})=\Delta_{|I_2|,S_2} \times \Delta_{|I_3|,S_3} \times \cdots \Delta_{|I_l|,S_l} \times [0,1]^{S_1}.
	\]
\end{proof}

\vspace{0.2cm}
Using the same setup as Proposition \ref{prop8} and combining with Proposition \ref{prop4}, we can give the concrete expression of the concave extension for the Schubert matroid rank function.

\begin{rmk}
	\label{rmk} The facet-defining inequalities of the polyhedron $\pi(\widetilde{\Pi}_i)$ can be written explicitly. For example, in (1) of Proposition \ref{prop8}, we suppose that $i \neq 1$, then
	\[
	\widetilde{\Pi}_i=
	\begin{dcases}
		\widetilde{x}=(\widetilde{x}_1, \cdots,\widetilde{x}_{2n}) \in \mathbb{R}^{2n}, \\
		\sum\limits_{m=1}^{2n}\widetilde{x}_m=n,\\
		\widetilde{x}_{\widetilde{S}_i} \geq |\widetilde{I}_i|,\\
		\widetilde{x}_{\widetilde{S}_i \cup \widetilde{S}_{i+1}} \geq |\widetilde{I}_i|+|\widetilde{I}_{i+1}|,\\
		\vdots\\
		\widetilde{x}_{\widetilde{S}_i \cup \cdots \cup \widetilde{S}_l} \geq |\widetilde{I}_i|+ \cdots +|\widetilde{I}_l|,\\
		\widetilde{x}_{\widetilde{S}_{i-1}} \leq |\widetilde{I}_{i-1}|,\\
		\widetilde{x}_{\widetilde{S}_{i-2} \cup \widetilde{S}_{i-1}} \leq |\widetilde{I}_{i-2}|+|\widetilde{I}_{i-1}|,\\
		\vdots\\
		\widetilde{x}_{\widetilde{S}_2 \cup \cdots \cup \widetilde{S}_{i-1}} \leq |\widetilde{I}_2|+ \cdots +|\widetilde{I}_{i-1}|.
	\end{dcases}
	\xrightarrow{\displaystyle \pi} \quad \pi(\widetilde{\Pi}_i)=
	\begin{dcases}
		x=(x_1,\cdots,x_n) \in \mathbb{R}^{n},\\
		x_{S_i}  \geq |I_i|,\\
		x_{S_i \cup S_{i+1}} \geq |I_i|+|I_{i+1}|,\\
		\vdots\\
		x_{S_i \cup \cdots S_l} \geq |I_i|+\cdots +|I_l|,\\
		x_{S_{i-1}} \leq |I_{i-1}|,\\
		x_{S_{i-2} \cup S_{i-1}} \leq |I_{i-2}|+|I_{i-1}|,\\
		\vdots\\
		x_{S_2 \cup \cdots S_{i-1}} \leq |I_2|+\cdots +|I_{i-1}|.
	\end{dcases}
	\]
	Notice that with the condition $\sum\limits_{m=1}^{2n}\widetilde{x}_m=n$, the inequality $\widetilde{x}_{\widetilde{S}_i \cup \cdots \cup \widetilde{S}_{j}} \geq |\widetilde{I}_i|+ \cdots +|\widetilde{I}_j|$ is equivalent to that $\widetilde{x}_{\widetilde{S}_{j+1} \cup \cdots \cup \widetilde{S}_{i-1}} \leq |\widetilde{I}_{j+1}|+ \cdots +|\widetilde{I}_{i-1}|$. Since $[n+1,2n] \subseteq \widetilde{S}_1$, we adopt one of these equivalent expressions to avoid having $\widetilde{S}_1$ to appear in the subscript. So we do not need to change the form of these inequalities under the projection $\pi$.
\end{rmk}

\vspace{0.2cm}
\begin{thm}
	\label{thm7}
	Corresponding to the cases of Proposition \ref{prop8}, the concave extension $\widehat{\theta_I}(x)$ can be written as
	\begin{enumerate}
		\item When $\sigma(\mathbf{pad}(I))=\sigma(I)=l$, i.e., $1 \in I$ and $n \notin I$,\\
		\[
		\widehat{\theta_I}(x)=
		\begin{cases}
			\sum\limits_{m=1}^{E(S_{i-1})}x_m+|I_i|+\cdots |I_l| &\text{for}\;x \in \pi(\widetilde{\Pi}_i) \cap [0,1]^n\;\text{and} \;i \in[2,l],\\
			\sum\limits_{m=1}^{E(S_{l})}x_m &\text{for}\;x \in \pi(\widetilde{\Pi}_1) \cap [0,1]^n\;\text{and} \;i=1.
		\end{cases}
		\]\\
		
		\item  When $\sigma(\mathbf{pad}(I))=\sigma(I)+1=l+1$,
		\begin{enumerate}
			\item If $1 \in I$ and $n \in I$, then\\
			\[
			\widehat{\theta_I}(x)=
			\begin{cases}
				\sum\limits_{m=1}^{E(S_{i-1})}x_m+|I_i|+\cdots |I_l|+|\widetilde{I}_{l+1}| &\text{for}\;x \in \pi(\widetilde{\Pi}_i) \cap [0,1]^n\;\text{and} \;i \in[2,l+1],\\
				\sum\limits_{m=1}^{n}x_m &\text{for}\;x \in \pi(\widetilde{\Pi}_1) \cap [0,1]^n\;\text{and} \;i=1.
			\end{cases}
			\]\\
			
			\item If $1 \notin I$ and $n \in I$, then\\
			\[
			\widehat{\theta_I}(x)=
			\begin{cases}
				\sum\limits_{m=1}^{E(S_{i-1})}x_m+|I_1 \backslash [1,E(S_1)]|+|I_i|+\cdots |I_l| &\text{for}\;x \in \pi(\widetilde{\Pi}_i) \cap [0,1]^n\;\text{and} \;i \in[2,l+1],\\
				k &\text{for}\;x \in \pi(\widetilde{\Pi}_1) \cap [0,1]^n\;\text{and} \;i=1.
			\end{cases}
			\]\\
			
			\item If $1 \notin I$ and $n \notin I$, then\\
			\[
			\widehat{\theta_I}(x)=
			\begin{cases}
				\sum\limits_{m=1}^{E(S_{i-1})}x_m+|I_i|+\cdots |I_l| &\text{for}\;x \in \pi(\widetilde{\Pi}_i) \cap [0,1]^n\;\text{and} \;i \in[2,l+1],\\
				k &\text{for}\;x \in \pi(\widetilde{\Pi}_1) \cap [0,1]^n\;\text{and} \;i=1.
			\end{cases}
			\]
		\end{enumerate}
	\end{enumerate}
\end{thm}

\begin{proof}
	We provide the proof of (1) and the rest follows similarly. By Proposition \ref{prop4}, we write down the formula of $\widehat{\theta}_{\mathbf{pad}(I)}(\widetilde{x})$ where $\widetilde{x}=(\widetilde{x}_1, \cdots, \widetilde{x}_{2n}) \in \Delta_{n,2n}$,
	\[
	\widehat{\theta}_{\mathbf{pad}(I)}(\widetilde{x})=
	\begin{cases}
		\sum\limits_{m=1}^{E(\widetilde{S}_{i-1})}\widetilde{x}_m+|\widetilde{I}_1 \backslash [1,E(\widetilde{S}_1)]|+|\widetilde{I}_i|+\cdots |\widetilde{I}_l| &\text{for}\;\widetilde{x} \in \widetilde{\Pi}_i \cap \Delta_{n,2n}\;\text{and} \;i \in[2,l],\\
		\sum\limits_{m=1}^{E(\widetilde{S}_{l})}\widetilde{x}_m+|\widetilde{I}_1 \backslash [1,E(\widetilde{S}_1)]| &\text{for}\;\widetilde{x} \in \widetilde{\Pi}_1 \cap \Delta_{n,2n}\;\text{and} \;i=1.
	\end{cases}
	\]\\
	
	By (1) of Proposition \ref{prop8}, we know that $|\widetilde{I}_1 \backslash [1,E(\widetilde{S}_1)]|=n-k$. Then according to Lemma \ref{lem5}, $\widehat{\theta}_I$ is the projection of $\widehat{\theta}_{\mathbf{pad}(I)}$ under $\pi$ with a translation $n-k$. Thus, we complete the proof of (1) 
\end{proof}

\vspace{1cm}
\subsection{Generalized matroid subdivisions from weakly separated sets}\quad\\
\indent The notion weak separation were introduced  by Leclerc and Zelevinsky (\cite{LZ}, 1998) to characterize all the quasi-commuting families of quantum minors in the q-deformed coordinate ring $\mathbf{Q}_q[\mathcal{F}]$ of type A flag variety $\mathcal{F}$.
\begin{defi}
	Given two sets $A,B \in 2^{[n]}$, we write $A \prec B$ if $a <b$ for all $a \in A$ and $b \in B$ (In particular if one or both of $A,B$ are empty). We say $A,B$ are \textbf{weakly separated} if at least one of the following two conditions holds:
	\begin{itemize}
		\item If $|A| \geq |B|$ and $B-A$ can be partitioned into a disjoint union $B-A=B' \sqcup B''$ such that $B' \prec A-B \prec B''$.
		\item If $|B| \geq |A|$ and $A-B$ can be partitioned into a disjoint union $A-B=A' \sqcup A''$ such that $A' \prec B-A \prec A''$.
	\end{itemize}
	In particular, if $B'$ or $B''$ (resp., $A'$ or $A''$) equals $\emptyset$, we say that $A,B$ are \textbf{strongly separated}. A collection of sets $\mathcal{C} \subseteq 2^{[n]}$ is called a \textbf{weakly separated set} if any two elements in $\mathcal{C}$ are weakly separated. When both $B'$ and $B''$ (resp., $A'$ and $A''$) are non-empty, we say that $B$ \textbf{surrounds} $A$ (resp.,$A$ surrounds $B$).
\end{defi}

When $A,B$ have the same size, i.e., $A,B \in \binom{[n]}{k}$ for some $k >0$, the weak separation is equivalent to the chord separation by Galashin \cite{G}.

\begin{defi}
	We say $A,B \in \binom{[n]}{k}$ are \textbf{chord separated} if there is no $1 \leq i<j<k<l \leq n$ such that $i,k$ belong to one of $A-B$ and $B-A$ while $j,l$ belong to the other.
\end{defi}

Early established a connection between the matroid subdivision of hypersimplex $\Delta_{k,n}$ and the weakly separated set of equal-sized sets in \cite{Ear}. For $I \in \binom{[n]}{k}$ and let $\beta_I$ be the  matroid subdivision of $\Delta_{k,n}$ induced by the translated blade, his conclusion says,

\begin{thm}[\cite{Ear}]
	\label{thm8}
	Let $\{I_1,\cdots,I_s\} \subseteq \binom{[n]}{k}$ be a collection of sets. The refinement of subdivisions $\beta_{I_1},\cdots, \beta_{I_s}$ is a  matroid subdivision of $\Delta_{k,n}$ if and only if $I_1,\cdots,I_s$ are weakly separated.
\end{thm}

Actually, there exists a discrete function version of Theorem \ref{thm8}, which regards the compatibility between $\beta_I$ and $\beta_J$ as whether the sum of corresponding M-concave functions remains a M-concave function.

 Inspired by the relationship between Schubert rank functions and blade arrangements discussed in Proposition \ref{prop9}, we derive Theorem \ref{g-positroid_ws}, 
a flag version of the above theorem. 

\begin{thm}
	\label{g-positroid_ws}
	Let $\{I_1,\cdots,I_s\} \subseteq 2^{[n]}$ be a collection of sets. The subdivision of $[0,1]^n$ induced by concave function $\widehat{\theta}:=\sum\limits_{j=1}^{s}\widehat{\theta}_{I_j}$ is a generalized matroid subdivision if and only if $I_1,\cdots,I_s$ are weakly separated. 
\end{thm}

 We leave the proof of this theorem to the next section. 
We will directly prove the compatibility of the tropical Pl\"ucker relations by using the crystal structure of the corresponding functions.

Let $Q(\theta)$ be the base polytope with support function $\theta$. Thus, weak separation is a necessary and sufficient condition for the Minkowski sum of Schubert matroid polytopes to be MV polytopes, as stated in the following result.

\begin{thm}
	\label{minkow_schubert}
	The Minkowski sum $Q(\theta)=\sum\limits_{j=1}^sQ(\theta_{I_j})$ is a MV polytope if and only if $I_1,\cdots,I_s$ are weakly separated.
\end{thm}
\begin{proof}
	This follows from a well-konwn fact, that is, $Q(f+g)=Q(f)+Q(g):=\{x+y\;|x \in Q(f),\;y \in Q(g)\}$ for any two submodular functions $f,g$.
\end{proof}

\vspace{0.2cm}
\begin{rmk}
	A similar conclusion also holds for dual Schubert matroid polytopes. For the so-called dual Schubert matroids $\Omega_I^{dual}$, as given in Definition \ref{lattice_path_matroid}, we take $J$ to be $[n-k+1,n]$, that is, $\Omega_I^{dual}:=M[I,[n-k+1,n]]$.
\end{rmk}

Similar to how Early expressed each point in the positive tropical Grassmannian as a weighted sum of translated blades (see \cite{Ear3}), based on the relationship between Schubert rank functions and translated blades, we conjecture here that every MV polytope can also be expressed as a Minkowski sum/difference of some Schubert matroid polytopes. The precise form is as follows.

\vspace{0.2cm}

\begin{conj}
	Suppose $Q$ is an MV polytope of type $A$. Then there exist  $\mathcal{C}_1, \mathcal{C}_2 \subseteq 2^{[n]}$ such that
	\[
	 Q=\sum\limits_{I \in \mathcal{C}_1}a_IQ(\theta_I)-\sum\limits_{I \in \mathcal{C}_2}b_IQ(\theta_I)
	\]
	where each $a_I$ and $b_I$ are positive.
\end{conj}

\vspace{1cm}

\subsection{Crystal structure on DCTP functions and proof of Theorem \ref{g-positroid_ws}}\quad\\
\label{CSODF}
The notion of crystals, introduced by Kashiwara \cite{Ka}, is a combinatorial object corresponding to representations of complex semisimple Lie algebra $\mathfrak{g}$.

\begin{defi}
	Let $\mathfrak{g}$ be a complex semisimple Lie algebra with weight lattice $\Lambda$ and simple roots $\alpha_1, \cdots, \alpha_n$. A \textbf{crystal} is a set $\mathcal{C}$ along with structure maps
	\[
	\mathbf{e}_i: \mathcal{C} \rightarrow \mathcal{C} \sqcup \{0\}, \;\;\;
	\mathbf{f}_i: \mathcal{C} \rightarrow \mathcal{C} \sqcup \{0\}, \;\; \text{for}\; i \in [n]\;\text{and}\;\mathbf{wt}: \mathcal{C} \rightarrow \Lambda,
	\]
	which satisfy the following axioms.
	\begin{enumerate}
		\item If $b \in \mathcal{C}$, $i \in [n]$ and $\mathbf{e}_i \cdot b \neq 0$, then $\mathbf{wt}(\mathbf{e}_i \cdot b)=\mathbf{wt}(b)+\alpha_i$.
		\item If $b \in \mathcal{C}$, $i \in [n]$ and $\mathbf{f}_i \cdot b \neq 0$, then $\mathbf{wt}(\mathbf{f}_i \cdot b)=\mathbf{wt}(b)-\alpha_i$.
		\item $b'=\mathbf{e}_i \cdot b$ if and only if $b=\mathbf{f}_i \cdot b'$.
	\end{enumerate}
\end{defi}

\indent In this section, all the DCTP functions $f$ are assumed to be $\mathbb{Z}$-valued. It has been proved by Kamnitzer in \cite{K} that DCTP functions with some normalized conditions can be endowed with a crystal structure isomorphic to $B(\infty)$ in type $A$, which can be formulated as follows:
\begin{prop}[\cite{K}]
	Let $f$ be a DCTP function, then
	\[
	(\mathbf{e_i}\cdot f)(X)=
	\begin{cases}
		\text{max}(f(X),f(X\backslash \{i+1\} \cup \{i\})-c) &\text{if } i \notin X \text{ and } i+1 \in X\\
		f(X) &\text{otherwise,}
	\end{cases}
	\]
	where $c=f([n]\backslash s_i[i])-f([n]\backslash [i])-1$ and $s_i$ is the transposition of $i$ and $i+1$.
\end{prop}

We say DCTP functions $f_1,f_2,\cdots,f_s$ are \textbf{compatible} if $f_1+f_2+\cdots f_s$ is a DCTP function. Clearly, by the tropical Pl\"ucker relation, compatibility condition is a pairwise condition i.e. $f_1,f_2,\cdots,f_s$ are compatible if and only if any two of them are compatible. Therefore, in the following discussion, we only need to consider the compatibility between the two DCTP functions.

Let $I_1,I_2 \in 2^{[n]}$ and $\mathcal{D}_{I_1,I_2}$ be a $2 \times n$ grid  where the shaded squares in the first row belong to $I_1$ while the shaded squares in the second row belong to $I_2$. For $i \in [n-1]$, $i$ is called an \textbf{ascent} if $I_k \cap \{i,i+1\} \neq \{i+1\}$ and let $s_i\mathcal{D}_{I_1,I_2}$ be the grid by swapping the columns $i$ and $i+1$ of $\mathcal{D}_{I_1,I_2}$. Continue performing swaps at ascents until the resulting grid has no more ascents. It is not difficult to see that this can be achieved through a finite number of swaps. Although the sequence of swaps is not unique, the resulting grid is uniquely determined. We denote the resulting grid by $\mathcal{D}_{I_1^*,I_2^*}$ where $|I_1|=|I_1^*|$ and $|I_2|=|I_2^*|$.

\vspace{0.3cm}
\begin{ex}
	\label{ex1}
	Let $n=6$ and $I_1=\{1,4,5\},\;I_2=\{1,3,4,6\}$, we let a sequence of swaps $s_5s_4s_3s_2s_1s_5s_4$ act on the grid $\mathcal{D}_{I_1,I_2}$ as Figure \ref{fig5},
	\begin{figure}[H]
		\centering
		\includegraphics[width=12cm]{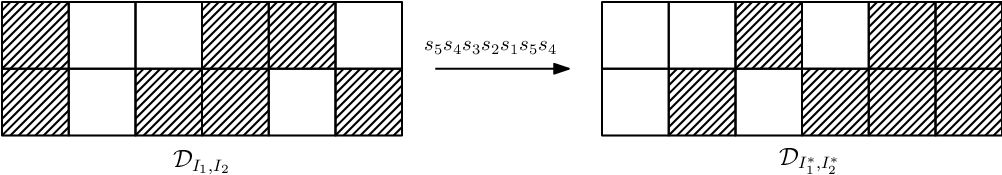}
		\caption{}
		\label{fig5}
	\end{figure}
	then $I_1^*=(s_5s_4s_3s_2s_1s_5s_4) \cdot I_1=\{3,5,6\}$ and $I_2^*=(s_5s_4s_3s_2s_1s_5s_4) \cdot I_2=\{2,4,5,6\}$.
\end{ex}

The following lemma can be directly derived from the definition of weak separation.
\begin{lem}
	\label{lem2}
	$I_1$ and $I_2$ are weakly separated if and only if $I_1^*$ and $I_2^*$ are weakly separated.
\end{lem}

In Sanchez's paper (\cite{S}, Theorem 6.8), the transposition operation performed at the ascent $i$ corresponds to the application of multiple raising operators to the DCTP function in the crystal structure. In our framework, we formulate it as follows.

\begin{thm}
	\label{thm4}
	Let $i$ be an ascent in $\mathcal{D}_{I_1,I_2}$, then
	\[
	\theta_{s_iI_1}+\theta_{s_iI_2}=\mathbf{e}_i^l \cdot (\theta_{I_1}+\theta_{I_2}),
	\]
	where l=1 or 2 denotes the number of $I_j\;(j=1,2)$ such that $i \in I_j$ but $i+1 \notin I_j$.
\end{thm}
Combining with Theorem \ref{thm4} and connectivity of crystal $B(\infty)$ in type A, we have the following corollary.

\begin{cor}
	\label{cor2}
	$\theta_{I_1}$ and $\theta_{I_2}$ are compatible if and only if $\theta_{I_1^*}$ and $\theta_{I_2^*}$ are compatible.
\end{cor}

The discussion of compatibility between $\theta_{I_1}$ and $\theta_{I_2}$ now reduces to examining the compatibility between $\theta_{I_1^*}$ and $\theta_{I_2^*}$. The advantage is that the configurations of $I_1^*$ and $I_2^*$ are more regular, making it easier to analyze the additivity of tropical Pl\"ucker relations.

$I_1^*$ and $I_2^*$ partition [n] into a disjoint union of intervals, that is, $[n]=A_0 \sqcup A_1 \sqcup \cdots \sqcup A_s$ with order $A_0 < A_1 < \cdots < A_s$. Among these intervals, $A_0$ belongs to neither $I_1^*$ nor $I_2^*$ while $A_s$ is in both $I_1^*$ and $I_2^*$ ($A_0$ and $A_s$ might be empty sets). $A_1, \cdots, A_{s-1}$ alternately lie in $I_1^*$ and $I_2^*$ i.e. if $A_i \subseteq I_1^*$ and $A_i \cap I_2^* = \emptyset$, then $A_{i+1} \subseteq I_2^*$ and $A_{i+1} \cap I_1^*=\emptyset$. We call these intervals $A_i$ \textbf{sections} of the grid $\mathcal{D}_{I_1^*,I_2^*}$ in the following discussion.

\begin{prop}
	\label{prop5}
	If $\theta_{I_1}$ and $\theta_{I_2}$ are compatible, then $I_1$ and $I_2$ are weakly separated
\end{prop}

\begin{proof}
	Suppose that $I_1$ and $I_2$ are not weakly separated, by Lemma \ref{lem2} and Corollary \ref{cor2}, we only need to prove that $\theta_{I_1^*}$ and $\theta_{I_2^*}$ are not compatible. Since $I_1^*$ and $I_2^*$ are not weakly separated, there exsits sections $A_{i-1},\; A_i,\; A_{i+1}$ such that $|A_{i-1}|+|A_{i+1}| >|A_i|$. Without loss of generality, we suppose that $A_i$ belongs to $I_1^*$ and  $A_{i-1},\;A_{i+1}$ belong to $I_2^*$. With inequality $|A_{i-1}|+|A_{i+1}| >|A_i|$, we can partition sections $A_{i-1}=A_{i-1}^1 \sqcup A_{i-1}^2$, $A_{i}=A_{i}^1 \sqcup A_{i}^2$ and $A_{i+1}=A_{i+1}^1 \sqcup A_{i+1}^2$ such that $|A_{i-1}^1|=x$, $|A_{i}^2|=x+1$, $|A_{i}^1|=y$ and $|A_{i+1}^2|=y+1$ ($x,y$ might be $0$) as in Figure \ref{fig6}. 
	\begin{figure}[H]
		\centering
		\includegraphics[width=13cm]{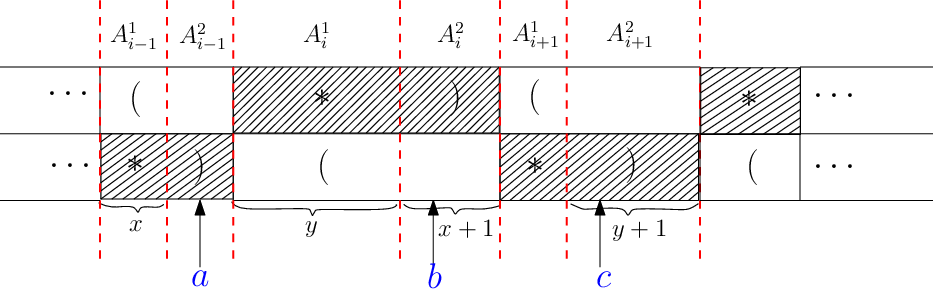}
		\caption{}
		\label{fig6}
	\end{figure}
	Let $X=(A_{i-1}^1 \sqcup A_{i}^1 \sqcup A_{i+1}^1) \sqcup A_{i+2} \sqcup \cdots A_{s}$ and take any $a \in A_{i-1}^2$, $b \in A_i^2$ and $c \in A_{i+1}^2$. Then 
	\[
	\begin{split}
		&\theta_{I_1^*}(Xab)=\theta_{I_1^*}(X)+1\quad \theta_{I_1^*}(Xa)=\theta_{I_1^*}(X)+1\quad \theta_{I_1^*}(Xb)=\theta_{I_1^*}(X)+1\\
		&\theta_{I_2^*}(Xbc)=\theta_{I_2^*}(X)+1 \quad
		\theta_{I_2^*}(Xb)=\theta_{I_2^*}(X)+1\quad \theta_{I_2^*}(Xc)=\theta_{I_2^*}(X)+1
	\end{split}
	\]
	Therefore
	\[
	\theta_{I_1^*}(Xac)+\theta_{I_1^*}(Xb)=\theta_{I_1^*}(Xbc)+\theta_{I_1^*}(Xa)=2\theta_{I_1^*}(X)+2>\theta_{I_1^*}(Xab)+\theta_{I_1^*}(Xc)=2\theta_{I_1^*}(X)+1,
	\]
	and
	\[
	\theta_{I_2^*}(Xac)+\theta_{I_2^*}(Xb)=\theta_{I_2^*}(Xab)+\theta_{I_2^*}(Xc)=2\theta_{I_2^*}(X)+3>\theta_{I_2^*}(Xbc)+\theta_{I_2^*}(Xa)=2\theta_{I_2^*}(X)+2.
	\]
	This implies that $\theta_{I_1^*}$ and $\theta_{I_2^*}$ are incompatible.
\end{proof}

Proposition \ref{prop5} provides a necessary condition for the compatibility of $\theta_{I_1}$ and $\theta_{I_2}$. Conversely, when $I_1$ and $I_2$ are weakly separated, can it be guaranteed that $\theta_{I_1}$ and $\theta_{I_2}$ are compatible under addition?

To simplify the notation, we define $\Phi_i(X,a,b):=\theta_{I_i^*}(Xa)+\theta_{I_i^*}(Xb)-\theta_{I_i^*}(Xab)-\theta_{I_i^*}(X)$ for $i=1,2$ and $a,b \in [n]\backslash X$. The following lemma can be directly derived from the properties of the matroid's rank function.

\begin{lem}
	\label{lem3}
	$\Phi_i(X,a,b) \in \{0,1\}$ for any $X \subseteq [n]$ and $a,b \in [n]\backslash X$. Besides,  $\Phi_i(X,a,b)=1$ if and only if $\theta_{I_i^*}(Xa)=\theta_{I_i^*}(Xb)=\theta_{I_i^*}(Xab)=\theta_{I_i^*}(X)+1$.
\end{lem}

When $I_i^*=[p,q] \sqcup [h,n]$ be the union of two intervals where $h > q+1$. Then $[n]=L_1 \sqcup L_1 \sqcup L_3 \sqcup L_4$ where $L_1=[1,p-1]$, $L_2=[p,q]$, $L_3=[q+1, h-1]$ and $L_4=[h,n]$ (Notice that $L_1=\emptyset$ if $p=1$). The following lemma classifies all the necessary cases that $\Phi_i(X,a,b)=1$ when $I_i^*$ is in the case described above.

\begin{lem}
	\label{lem4} 
	The following statements are equivalent when $a$ and $b$ are in different positions.
	\begin{enumerate}
		\item When $a,b$ are shown in Figure \ref{subfig7} and \ref{subfig8},  $\Phi_i(X,a,b)=1$ is equivalent to that
		\[
		\max\{|L_4\backslash X|-|L_3 \cap X|, 0\}+|L_2 \backslash X|=|L_1 \cap X|+1.
		\]
		\item When $a,b$ are shown in Figure \ref{subfig9} and \ref{subfig10},  $\Phi_i(X,a,b)=1$ is equivalent to that
		\[
		\max\{|L_1 \cap X|-|L_2 \backslash X|, 0\}+|L_3 \cap X|=|L_4 \backslash X|-1.
		\]
		\item When $a,b$ are shown in Figure \ref{subfig11}, \ref{subfig12}, \ref{subfig13} and \ref{subfig14},  $\Phi_i(X,a,b)=1$ is equivalent to that 
		\[
		|L_4\backslash X|-|L_3 \cap X|>0\;\;and \;\;|L_4\backslash X|-|L_3 \cap X|+|L_2 \backslash X|=|L_1 \cap X|+1.
		\]
	\end{enumerate}
	\begin{figure}[htbp]
		\raggedright
		\begin{subfigure}[t]{0.3\textwidth}
			\includegraphics[width=4.5cm]{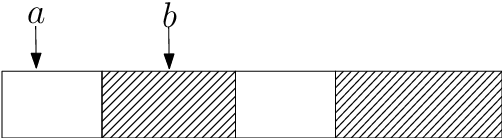}
			\caption{$a \in L_1, \;b \in L_2$}
			\label{subfig7}
		\end{subfigure}
		\hfill
		\begin{subfigure}[t]{0.3\textwidth}
			\includegraphics[width=4.5cm]{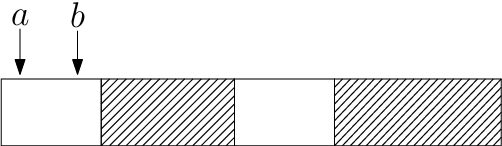}
			\caption{$a, b \in L_1$}
			\label{subfig8}
		\end{subfigure}
		\hfill
		\begin{subfigure}[t]{0.3\textwidth}
			\includegraphics[width=4.5cm]{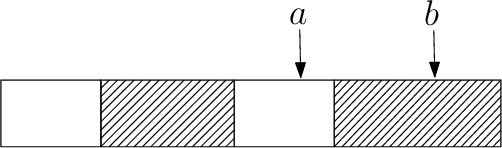}
			\caption{$a\in L_3,\;b \in L_4$}
			\label{subfig9}
		\end{subfigure}
		\vspace{0.5cm}
		\begin{subfigure}[t]{0.3\textwidth}
			\includegraphics[width=4.5cm]{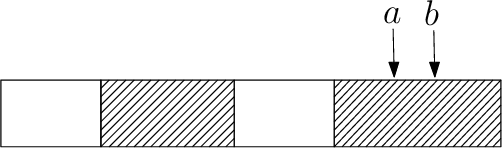}
			\caption{$a,b\in L_4$}
			\label{subfig10}
		\end{subfigure}
		\hfill
		\begin{subfigure}[t]{0.3\textwidth}
			\includegraphics[width=4.5cm]{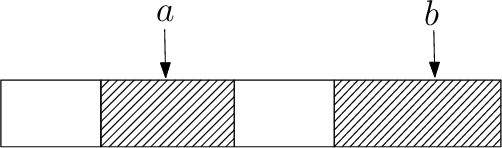}
			\caption{$a\in L_2,\;b \in L_4$}
			\label{subfig11}
		\end{subfigure}
		\hfill
		\begin{subfigure}[t]{0.3\textwidth}
			\includegraphics[width=4.5cm]{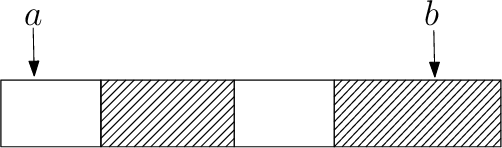}
			\caption{$a\in L_1,\;b \in L_4$}
			\label{subfig12}
		\end{subfigure}
		\vspace{0.5cm}
		\begin{subfigure}[t]{0.3\textwidth}
			\includegraphics[width=4.5cm]{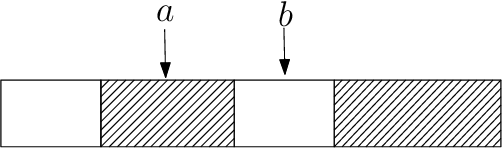}
			\caption{$a\in L_2,\;b \in L_3$}
			\label{subfig13}
		\end{subfigure}
		\hspace{0.5cm}
		\begin{subfigure}[t]{0.3\textwidth}
			\includegraphics[width=4.5cm]{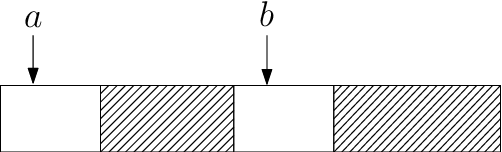}
			\caption{$a \in L_1,\;b \in L_3$}
			\label{subfig14}
		\end{subfigure}
		\hfill
		\caption{}
		\label{fig714}
	\end{figure}
\end{lem}

\begin{proof}
	For the three cases above, we prove one instance of each, as the others follow analogously. By Lemma \ref{lem3}, $\Phi_i(X,a,b)=1$ is equivalent to that $\theta_{I_i^*}(Xa)=\theta_{I_i^*}(Xb)=\theta_{I_i^*}(Xab)=\theta_{I_i^*}(X)+1$. Thus, let us apply the method of computing $\theta_{I_i^*}(X)$ in the beginning of Section \ref{SMABA}.
	
	(1). In the case (A) of Figure \ref{fig714}, $|L_4\backslash X|$ denotes the number of )'s in the section $L_4$ while $|L_3 \cap X|$ denotes the number of ('s in the section $L_3$. Thus, $\max\{|L_4\backslash X|-|L_3 \cap X|,0\}$ denotes the number of unpaired )'s in the sections $L_3$ and $L_4$ while computing $\theta_{I_i^*}(X)$. Since $\theta_{I_i^*}(Xa)=\theta_{I_i^*}(X)+1$, the total number of unpaired )'s in $L_3\sqcup L_4$ and )'s in $L_2$ is bigger than ('s in $L_1$, i.e.
	\begin{equation}
		\label{E1}
		\max\{|L_4\backslash X|-|L_3 \cap X|, 0\}+|L_2 \backslash X|>|L_1 \cap X|.
	\end{equation}
	Therefore, there are $max\{|L_4\backslash X|-|L_3 \cap X|, 0\}+|L_2 \backslash X|-|L_1 \cap X|-1$ unpaired )'s in computing $\theta_{I_i^*}(Xa)$. On the other hand, $\theta_{I_i^*}(Xa)=\theta_{I_i^*}(Xab)$ implies that transforming the ) to $*$ in the box $b$ keeps $\theta_{I_i^*}(Xa)$ unchanged. This shows that the )
	in box $b$ has been paired which implies that there is no unpaired )'s in computing $\theta_{I_i^*}(Xa)$, i.e.
	\begin{equation}
		\label{E2}
		\max\{|L_4\backslash X|-|L_3 \cap X|, 0\}+|L_2 \backslash X|-|L_1 \cap X|-1 \leq 0.
	\end{equation}
	Combining equations \ref{E1} and \ref{E2}, we complete the proof. The inverse direction can be checked direcly by showing that $\theta_{I_i^*}(Xa)=\theta_{I_i^*}(Xb)=\theta_{I_i^*}(Xab)=\theta_{I_i^*}(X)+1$.
	
	(2). Similar to (1). In the case (C) of Figure \ref{fig714}, the unpaired ('s in $L_1 \sqcup L_2$ is $max\{|L_1 \cap X|-|L_2\backslash X|,0\}$ when computing $\theta_{I_i^*}(X)$. Since $\theta_{I_i^*}(Xb)=\theta_{I_i^*}(X)+1$, the total number of unpaired ('s in $L_1 \sqcup L_2$ and ('s in $L_3$ is less than )'s in $L_4$, i.e.
	\begin{equation}
		\label{E3}
		max\{|L_1 \cap X|-|L_2\backslash X|,0\}+|L_3 \cap X|<|L_4\backslash X|.
	\end{equation}
	Similarly while computing $\theta_{I_i^*}(Xa)$, the total number of unpaired ('s in $L_1 \sqcup L_2$ and ('s in $L_3$ is $\max\{|L_1 \cap X|-|L_2\backslash X|,0\}+|L_3 \cap X|+1$. Since $\theta_{I_i^*}(Xab)=\theta_{I_i^*}(Xb)$, this shows that adding one ( in the box $b$ dose not change the value of $\theta_{I_i^*}(Xa)$, i.e. the adding ( in box $b$ is unpaired. Thus
	\begin{equation}
		\label{E4}
		\max\{|L_1 \cap X|-|L_2\backslash X|,0\}+|L_3 \cap X|+1 \geq |L_4\backslash X|.
	\end{equation}
	Combining equations \ref{E3} and \ref{E4}, we complete the proof.
	
	(3). In the case (E) of Figure \ref{fig714}, if $|L_4\backslash X|-|L_3 \cap X| \leq 0$, then $\theta_{I_i^*}(Xb)=\theta_{I_i^*}(X)$ which contradicts to the assumption $\theta_{I_i^*}(Xb)=\theta_{I_i^*}(X)+1$. So $\max\{|L_4\backslash X|-|L_3 \cap X|,0\}=|L_4\backslash X|-|L_3 \cap X|$ and by the similar arguments in (1) and (2), $\Phi_i(X,a,b)=1$ requires the additional condition
	\begin{equation}
		\label{E5}
		|L_4\backslash X|-|L_3 \cap X|+|L_2 \backslash X|=|L_1 \cap X|+1.
	\end{equation}
\end{proof}

Sanchez has proved that $\theta_{I_1}$ and $\theta_{I_2}$ are compatible when $I_1$ and $I_2$ are strongly separated in \cite{S}. So here, in order to prove the inverse side of Theorem \ref{g-positroid_ws}, we only need to prove for the case that $I_1$ and $I_2$ are weakly separated but not strongly separated. The contrapositive of this proposition is stated as follows.

\begin{prop}
	\label{prop6}
	Suppose that $I_2\backslash I_1$ surrounds $I_1 \backslash I_2$. If $\theta_{I_1}$ and $\theta_{I_2}$ are incompatible, then $I_1$ and $I_2$ are not weakly separated.
\end{prop}

\begin{proof}
	By Lemma \ref{lem2} and Corollary \ref{cor2}, we can reduce $I_1$ and $I_2$ to consider $I_1^*$ and $I_2^*$.
	Since $I_2^*\backslash I_1^*$ surrounds $I_1^* \backslash I_2^*$, $[n]$ can be decomposed into 5 sections $A_0, \cdots,A_4$ such that $I_1^*=A_2\sqcup A_4$ and $I_2^*=A_1 \sqcup A_3 \sqcup  A_4$ ($A_0$ and $A_4$ might be empty).
	
	Meanwhile, by Lemma \ref{lem3}, $\theta_{I_1^*}$ and $\theta_{I_2^*}$ are incompatible if and only if $\exists \;X \subseteq [n]$ and $a,b,c \in [n]\backslash X$ with  $a<b<c$ such that one of the following two conditions holds:
	\begin{equation}
		\Phi_1(X,a,c)=\Phi_2(X,a,c)=0\;\; \text{and}\;\; \Phi_1(X,a,b)=\Phi_2(X,b,c)=1 \tag{$\clubsuit \; .1$}
	\end{equation}
	\begin{equation}
		\Phi_1(X,a,c)=\Phi_2(X,a,c)=0\;\; \text{and}\;\; \Phi_1(X,b,c)=\Phi_2(X,a,b)=1 \tag{$\clubsuit \; .2$}
	\end{equation}
	
	First, it is easy to see that for condition ($\clubsuit \; .1$) or ($\clubsuit \; .2$) to hold, any two elements among $a, b$ and $c$ cannot fall within the same section. Thus, there are $\binom{5}{3}=10$ possible arrangements of $a, b$ and $c$ as shown in Figure \ref{fig1524}. We will show that most of these 10 scenarios are impossible. For a thorough proof, we address each case individually.
	
	\begin{figure}[H]
		\centering
		\begin{subfigure}[t]{0.3\textwidth}
			\includegraphics[width=4.5cm]{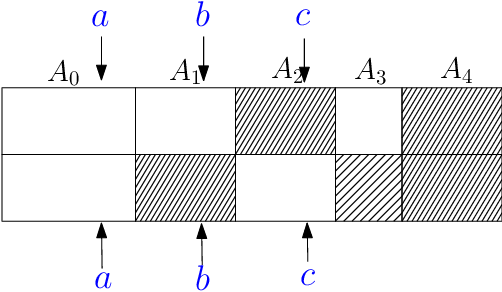}
			\caption*{Case 1: $a \in A_0,\;b \in A_1,\;c\in A_2$}
			\label{subfig15}
		\end{subfigure}
		\hfill
		\begin{subfigure}[t]{0.3\textwidth}
			\includegraphics[width=4.5cm]{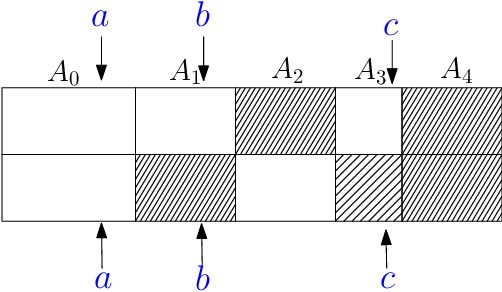}
			\caption*{Case 2: $a \in A_0,\;b \in A_1,\;c\in A_3$}
			\label{subfig16}
		\end{subfigure}
		\hfill
		\begin{subfigure}[t]{0.3\textwidth}
			\includegraphics[width=4.5cm]{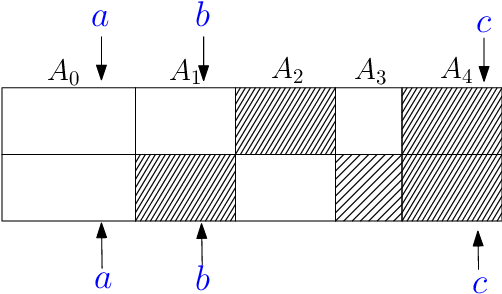}
			\caption*{Case 3: $a \in A_0,\;b \in A_1,\;c\in A_4$}
			\label{subfig17}
		\end{subfigure}
		\vspace{10pt}
		\begin{subfigure}[t]{0.3\textwidth}
			\includegraphics[width=4.5cm]{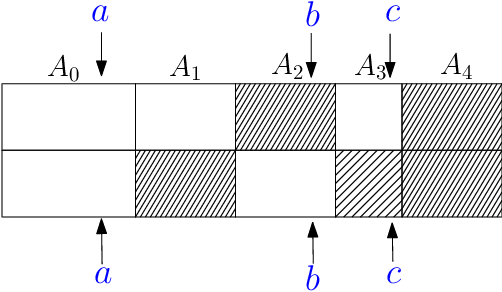}
			\caption*{Case 4: $a \in A_0,\;b \in A_2,\;c\in A_3$}
			\label{subfig18}
		\end{subfigure}
		\hfill
		\begin{subfigure}[t]{0.3\textwidth}
			\includegraphics[width=4.5cm]{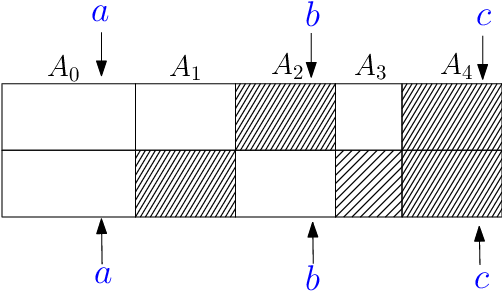}
			\caption*{Case 5: $a \in A_0,\;b \in A_2,\;c\in A_4$}
			\label{subfig19}
		\end{subfigure}
		\hfill
		\begin{subfigure}[t]{0.3\textwidth}
			\includegraphics[width=4.5cm]{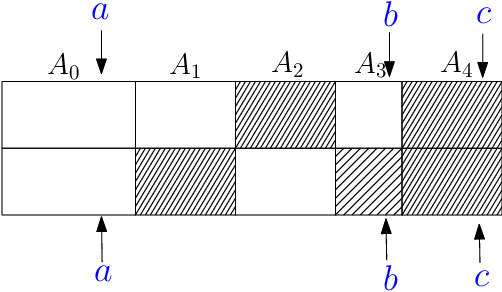}
			\caption*{Case 6: $a \in A_0,\;b \in A_3,\;c\in A_4$}
			\label{subfig20}
		\end{subfigure}
		
		\begin{subfigure}[t]{0.3\textwidth}
			\includegraphics[width=4.5cm]{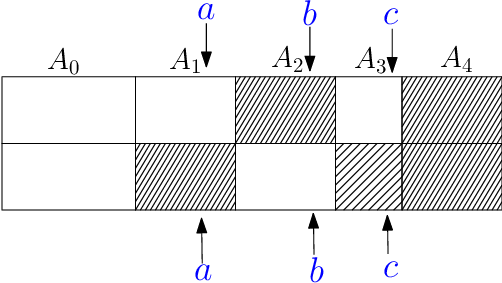}
			\caption*{Case 7: $a \in A_1,\;b \in A_2,\;c\in A_3$}
			\label{subfig21}
		\end{subfigure}
		\hfill
		\begin{subfigure}[t]{0.3\textwidth}
			\includegraphics[width=4.5cm]{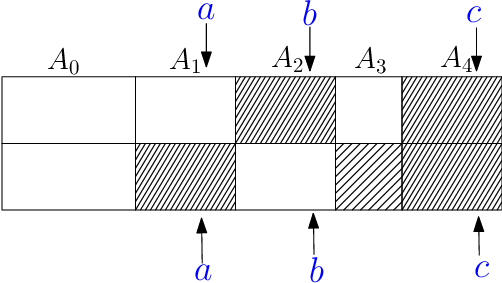}
			\caption*{Case 8: $a \in A_1,\;b \in A_2,\;c\in A_4$}
			\label{subfig22}
		\end{subfigure}
		\hfill
		\begin{subfigure}[t]{0.3\textwidth}
			\includegraphics[width=4.5cm]{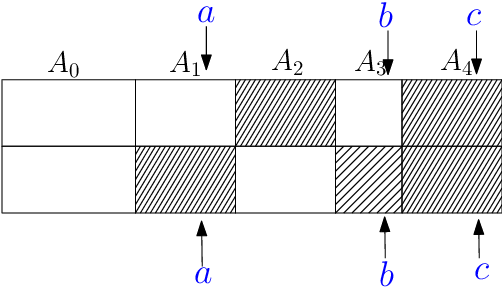}
			\caption*{Case 9: $a \in A_1,\;b \in A_3,\;c\in A_4$}
			\label{subfig23}
		\end{subfigure}
		
		\begin{subfigure}[t]{0.3\textwidth}
			\includegraphics[width=4.5cm]{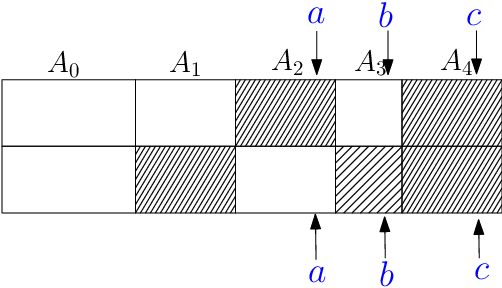}
			\caption*{Case 10: $a \in A_2,\;b \in A_3,\;c\in A_4$}
			\label{subfig24}
		\end{subfigure}
		\caption{}
		\label{fig1524}
	\end{figure}
	
	\noindent Case 1: By (1) of Lemma \ref{lem4}, we have $\Phi_1(X,a,c)=1 \Leftrightarrow \Phi_1(X,b,c)=1$. Thus  condition ($\clubsuit \; .2$) is not satisfied. On the other hand, by (2) of Lemma \ref{lem4}, $\Phi_2(X,a,c)=1 \Leftrightarrow \Phi_2(X,b,c)=1$, so condition ($\clubsuit \; .1$) is not satisfied. This case is impossible. Similar to Case 1, cases 2, 3, 6, 9, 10 are ruled out for the same reason.\\
	
	\noindent Case 4: By (3) of Lemma \ref{lem4}, $\Phi_1(X,a,c)=1 \Leftrightarrow \Phi_1(X,b,c)=1$ and $\Phi_2(X,a,c)=1 \Leftrightarrow \Phi_2(X,a,b)=1$. Therefore, ($\clubsuit \; .1$) is the only possible case for $\theta_{I_1^*}$ and $\theta_{I_2^*}$ to be incompatible. Thus by (1) in Lemma \ref{lem4}, $\Phi_1(X,a,b)=1$ if and only if
	\begin{equation}
		\label{E6}
		\max\{|A_4\backslash X|-|A_3 \cap X|,0\}+|A_2\backslash X|=|A_0 \cap X|+|A_1 \cap X|+1.
	\end{equation}
	And by (2) of Lemma \ref{lem4}, we have $\Phi_2(X,b,c)=1$ if and only if
	\begin{equation}
		\label{E7}
		 \max\{|A_0 \cap X|-|A_1 \backslash X|, 0\}+|A_2 \cap X|=|A_3\backslash X|+|A_4\backslash X|-1.
	\end{equation}
	Add equations \ref{E6} and \ref{E7}, we get the equation
	\begin{equation}
		\label{E8}
		|A_2|+\max\{|A_3\cap X|-|A_4\backslash X|,0\}+\max\{|A_1 \backslash X|-|A_0 \cap X|,0\}=|A_1|+|A_3|.
	\end{equation}
	We claim that $|A_1 \backslash X|-|A_0 \cap X|>0$, otherwise suppose that $|A_0 \cap X|-|A_1 \backslash X| \geq 0$, equation \ref{E6} can be writen as $|(A_3 \sqcup A_4)\backslash X|-|A_2 \cap X|+|A_1 \cap X|=|A_0 \cap X|+1$. Moreover, $|A_0 \cap X|-|A_1 \backslash X| \geq 0$ is equivalent to that $|(A_3 \sqcup A_4)\backslash X|-|A_2 \cap X|>0$. Thus this is equivalent to that $\Phi_2(X,a,c)=1$ by (3) of Lemma \ref{lem4} which contradicts the condition ($\clubsuit \; .1$). Therefore $|A_2|<|A_1|+|A_3|$ and this shows that $I_1^*$ and $I_2^*$ are not weakly separated. In cases 5, 7, 8, we follow the same arguments to show that $|A_2|<|A_1|+|A_3|$ as case 4. Then we complete the proof.
\end{proof}

\vspace{0.2cm}
\begin{proof}[Proof of Theorem \ref{g-positroid_ws}]
	The "only if" part follows from Proposition \ref{prop5}, that is,  if $\theta:=\sum_{I \in \mathcal{C}} c_I \theta_I$ is a DCTP function, then $\mathcal{C}$ is a weakly separated set. while the "if" part follows from Proposition \ref{prop6} and Theorem C in \cite{S}.
\end{proof}

\vspace{1cm}

\section{Generalized polypositroids}
\label{DODF}
\subsection{Generalized polypositroids and DCTP functions}\quad\\
In this section, we introduce a special class of g-polymatroids called g-polypositroids. It generalizes the positroids from \cite{ARW} and the polypositroids from \cite{LP}, and can also be used to characterize the affine regions of DCTP functions.

For a set \([n]\), denote \( \mathcal{I}(n) \) the set of interval subsets of \([n]\), of the form \([a, b] := \{a, a+1, \cdots, b\}\), then we use \( \mathcal{I}(n) \) to define the subclass.

\begin{defi}
	\label{def3}
	Let $(f, g)$ be a strong pair. The g-polymatroid $Q(f, g)$ is called a \textbf{generalized polypositroid} (abbreviated as \textbf{g-polypositroid}) if $Q(f, g) = QI(f, g)$, where
\begin{equation}
	QI(f,g):=\{x \in \mathbb{R}^n\;|\;g([a,b]) \leq x([a,b]) \leq f([a,b]), \; \forall \; [a,b] \in \mathcal{I}(n)\}.
\end{equation}
Denote by \( \mathcal{GPP}(n) \) the class of g-polypositroids in \( \mathbb{R}^n \).
\end{defi}

\vspace{0.2cm}
\begin{rmk}
In the above definition,
\begin{enumerate}
	\item when $g(X)=f([n])-f([n]\backslash X)$, then $Q(f,g)=\operatorname{QI}(f,g)$ becomes an alcoved polytope , which degenerates to the \textbf{polypositroid} defined in \cite{LP}. \cite{LP} employs various combinatorial models, such as Coxeter necklaces, to parameterize polypositroids. 
	
	\item In particular, when $Q$ is a polypositroid and for every $i \in [n]$, $f(\{i\})=0$ or $1$, then $Q$ degenerates to a \textbf{positroid}, and the corresponding Coxeter necklace degenerates to a Grassmann necklace.
\end{enumerate}
\end{rmk}

\vspace{0.3cm}

Definition \ref{def3} has equivalent characterizations. Given an inequality $x(S) \leq f(S)$ (or $x(S) \geq g(S)$), we call $S$ the \textbf{support} of this inequality. From Definition \ref{def3}, we know that all facet-defining inequalities have interval supports. 

Conversely, if $Q(f,g)$ is a g-polymatroid whose supports of facet-defining inequalities are intervals, then $Q$ is a  g-polypositroid. Here, we need a further explanation. In Theorem \ref{thm1}, we can express the affine hull of $Q$ as
\[
\text{Aff}(Q)=H_1 \times H_2 \times \cdots H_k \times \mathbb{R}^T,
\]
where $H_i=\{x \in \mathbb{R}^{S_i}\;|\;x(S_i)=g(S_i)=f(S_i)\}$ is a hyperplane in $\mathbb{R}^{S_i}$. Thus every facet-defining equality has many equivalent expressions by adding or substracting equations of the forms $x(S_i)=f(S_i)$ ($i \in [k]$). Therefore, here we refer to the assumption that among the equivalent expressions of each facet-defining inequality, there exists at least one whose support is an interval.

More explicitly, suppose that $Q$ is not a g-polypositroid, i.e., $Q(f,g) \subsetneq QI(f,g)$, so there exists a non-interval $A$ such that all the equivalent expressions of $x(A) \leq f(A)$ or $x(A) \geq g(A)$ have non-interval supports, and one of them is facet-defining. We summarize the above discussion as the following proposition.

\begin{prop}
    \label{prop12}
    A g-polymatroid is a g-polypositroid if and only if among each equivalent class of facet-defining inequalities, there exists at least one whose support is an interval.
\end{prop}

Analogues to g-polymatroids, g-polypositroids are also closed under some operations. 

\begin{prop}
	\label{prop13}
    The g-polypositroids in $\mathbb{R}^n$ are closed under the following operations.
    \begin{enumerate}
        \item Intersection with a plank $P(\alpha,\beta):=\{x \in \mathbb{R}^n\;|\;\alpha \leq x([n]) \leq \beta\}$, where $\alpha,\beta \in \mathbb{R}$.
        \item Intersection with a box $B(\alpha,\beta):=\{x \in \mathbb{R}^n\;|\;\alpha \leq x_i \leq \beta \; \text{for} \; i \in [n]\}$, where $\alpha,\beta \in \mathbb{R}$.
        \item Let $\pi_m:\mathbb{R}^n \rightarrow \mathbb{R}^m$ be the projection onto the first $m$ coordinates, then $\pi_m(Q)$ is a g-polypositroid in $\mathbb{R}^m$.
    \end{enumerate}
\end{prop}

\begin{proof}
	According to \cite{FT}, the intersection of Q with either a plank or a box is a g-polymatroid. From Proposition \ref{prop12}, we know that under both of these operations, the supports of the facet-defining inequalities remain intervals. Therefore, their intersection remains a g-polypositroid.
	
	As for the projection operation, it follows from \cite{FT} (Chapter 3) that the strong pair of the projected g-polymatroid $\pi_m(Q)$ is $(f_m, g_m)$, where $f_m$ and $g_m$ are the restrictions of $f$ and $g$ to $2^{[m]}$, respectively. Here we prove that $\pi_m(QI(f,g)) = QI(f_m, g_m)$, so from $\pi_m(Q(f,g)) = Q(f_m, g_m)$ and $Q(f,g) = QI(f,g)$, we obtain $\pi_m(Q(f,g)) = QI(f_m, g_m)$ is a g-polypositroid. To prove $\pi_m(QI(f,g)) = QI(f_m, g_m)$, we use induction on $m$. The case $m = n$ is obvious, we assume it holds for $m + 1$ and prove that $\pi_m(QI(f_{m+1},g_{m+1}))=QI(f_m,g_m)$. The inclusion "$\subseteq$" follows directly, we focus on the other direction.

	Given a $x_1 \in \mathbb{R}^m$ that lies in $QI(f_m,g_m)$, we find a $y \in \mathbb{R}$ such that the lifted point $x=(x_1,y) \in \mathbb{R}^{m+1}$ lies in $QI(f_{m+1},g_{m+1})$. Let
	\begin{align*}
		l&=\min\left\{f_{m+1}(I)-x_1(I\backslash \{m+1\})\;|\;m+1 \in I,\;I \in \mathcal{I}(m+1)\right\},\\
		u&=\max\left\{x_1(J\backslash \{m+1\}-g_{m+1}(J))\;|\;m+1 \in J,\;J \in \mathcal{I}(m+1)\right\}.
	\end{align*}
	For such $I, J \in \mathcal{I}(m+1)$, by strong pair condition, we have 
	\[
	f_{m+1}(I)-g_{m+1}(J) \geq f_{m+1}(I \backslash J)-g_{m+1}(J \backslash I) \geq x_1(I \backslash J)- x_1(J \backslash I)=x_1(I \backslash \{m+1\})-x_1(J \backslash \{m+1\}),
	\]
	thus we have $l \geq u$. Here we choose $y$ such that $l \geq y \geq u$. Next, we prove that $x$ lies within $QI(f_{m+1},g_{m+1})$. Take any $I \in \mathcal{I}(m+1)$, if $m+1 \notin I$, If $m+1 \notin I$, then by the assumption that $x_1 \in QI(f_m,g_m)$, we directly obtain $g_{m+1}(I) \leq x(I) \leq f_{m+1}(I)$. If $m+1 \in I$, then
	\begin{equation*}
		\begin{split}
			f_{m+1}(I)-x(I)&=f_{m+1}(I)-x_1(I\backslash \{m+1\})-y\\
			& \geq (f_{m+1}(I)-x_1(I\backslash \{m+1\}))-l \geq 0.
		\end{split}
	\end{equation*}
	Similarly, we obtain $g_{m+1}(I) \leq x(I)$. Thus, by the induction hypothesis, we have that $\pi_m(Q)=QI(f_m,g_m)$ is a g-polypositroid.
\end{proof}

\vspace{0.3cm}
\begin{rmk}
	From Proposition \ref{prop13}, it is not difficult to see that if we replace $\mathbb{R}^m$ in (3) with $\mathbb{R}^{[a,b]}$, where $[a,b]$ is an integer interval, the conclusion still holds.
\end{rmk}

\vspace{1cm}

\subsection{Minkowski decompositions for g-polymatroids/polypositroids into simplexes}\quad\\

Recall that $A+B:=\{a+b,\;a \in A,\;b \in B\}$ is the \textbf{Minkowski sum} of $A$ and $B \subseteq \mathbb{R}^n$, and 
\[
A-B:=\{c \in \mathbb{R}^n,\;c+B \subseteq A\}
\]
is called the \textbf{Minkowski difference} of $A$ and $B$. Note that Minkowski difference might be empty.

Let $f:2^{N} \rightarrow \mathbb{R}$ be a function and let $\mu_f$ be its M\"obius inversion. That means there holds
\begin{equation}
\label{E16}
f(S)=\sum\limits_{T \subseteq S}\mu_f(T), \qquad \mu_f(T)=\sum\limits_{S \subseteq T} (-1)^{|T|-|S|}f(S).
\end{equation}
For such a function $f:2^N \rightarrow \mathbb{R}$, define the polyhedron
\[
\Delta(f):=\{x \in \mathbb{R}^N\;|\;x(S) \geq f(S),\;x(N)=f(N)\}.
\]
Note that $\Delta(f)$ might be empty.

Since we are interested in decomposition of polyhedra into Minkowski sums and difference of simplexes, we can consider polyhedra up to translation by an integer vector. The latter means that w.l.o.g we may consider non-negative function $f$, since the translation of a polyhedron $\Delta(f)$ by an vector $z \in \mathbb{R}^N$ means adding a linear function,
\[
f(S) \rightarrow f(S)+\sum\limits_{s \in S}z_s.
\]
Note that the M\"obius inversion $\mu_z(S)=0$ if $|S|>1$ and $\mu_z(i)=z_i$ for $i \in N$.

In the paper \cite{DK2} by Danilov and the first author, they proved a formula for decomposing a base polyhedron into Minkowski sums and difference of simplexes.

\begin{thm}[\cite{DK2}]
    \label{thm12}
    There holds
    \begin{equation}
    \label{E10}
    \Delta(f)=\left(\sum\limits_{S,\;\mu_f(S)>0}\mu_f(S)\Delta_S\right)-\left(\sum\limits_{T,\;\mu_f(T)<0}(-\mu_f(T))\Delta_T\right),
    \end{equation}
    
    where $\Delta_S:=Conv\{e_i\;|\;i \in S\}$, $Conv\{\cdot\}$ denotes taking the convex hull. If $f:2^N \rightarrow \mathbb{R}_+$ is supermodular, then $\sum\limits_{T,\;\mu_f(T)<0}(-\mu_f(T))\Delta_T$ is a Minkowski summand of $\sum\limits_{S,\;\mu_f(S)>0}\mu_f(S)\Delta_S$.
    
    The latter implies
    \[
    \sum\limits_{S,\;\mu_f(S)>0}\mu_f(S)\Delta_S=\Delta(f)+\left(\sum\limits_{T,\;\mu_f(T)<0}(-\mu_f(T))\Delta_T\right).
    \]
\end{thm}

\vspace{0.3cm}
\begin{rmk}
	For a polytope 
	\[
	\nabla(g):=\{x\in \mathbb R^N \; |\; x(S)\le g(S), \, x(N)=g(N)\}.
	\]
	we can see that, for $f(S)=g(N)-g(N-S)$, we have
	\[
	\nabla (g)=\Delta(f).
	\] 
	Therefore we get a similar theorem for 
	$\nabla(g)$ applying the M\"obius inversion $\mu_g$.
\end{rmk}
\vspace{0.5cm}
According to Fujishige's notes \cite{F} , every g-polymatroid can be represented as a projection of some base polyhedron associated with a sub/supermodular system.

Given a strong pair $(g,f)$ on the domain $2^N$, define the extended base set $\widehat{N}:=N \cup \{\hat{e}\}$ by adding a new element $\hat{e} \notin N$. Then we can construct a new supermodular function $h:2^{\widehat{N}} \rightarrow \mathbb{R}$ by setting
\begin{equation}
\label{E12}
h(X):=g(X)\;\; \text{and}\;\;h(X\hat{e}):=f(N)-f(N-X)\;\;\text{for every}\;\;X\in 2^N 
\end{equation}

\vspace{0.3cm}
\begin{thm}[\cite{F}]
    \label{thm13}
    Then g-polymatroid $Q(f,g)$ is the projection of the base polyhedron $Q(h,h^{\natural})$ by setting $x(\hat{e})=0$, where
    \[
    Q(h,h^{\natural}):=\{x \in \mathbb{R}^{\widehat{N}}\;|\;x(S) \geq h(S)\;\forall\; S \subsetneqq \widehat{N}\:\:\text{and} \;\;x(\widehat{N})=h(\widehat{N})\}
    \]
\end{thm}

\vspace{0.3cm}
Combining Theorem \ref{thm12} and \ref{thm13}, we obtain the decomposition formula for g-polymatroids.

\begin{thm}
    \label{thm14}
    There holds
    \begin{equation}
    \label{E11}
    \begin{split}
        Q(f,g)&=\left[\left(\sum\limits_{S,\;(\mu_{f^{\natural}}-\mu_g)(S)>0}(\mu_{f^{\natural}}-\mu_g)(S)\overline{\Delta}_S\right)-\left(\sum\limits_{T,\;(\mu_{f^{\natural}}-\mu_g)(T)<0}(\mu_g-\mu_{f^{\natural}})(T)\overline{\Delta}_T\right)\right]\\
        &+\left[\left(\sum\limits_{S,\;\mu_g(S)>0}\mu_g(S)\Delta_S\right)-\left(\sum\limits_{T,\;\mu_g(T)<0}(-\mu_g(T))\Delta_T\right)\right]
    \end{split}
    \end{equation}
    for any strong pair $(f,g)$ over domain $2^N$, where $S \subseteq N$ and  $\overline{\Delta}_S:=Conv\{0,\;\Delta_S\}$.
\end{thm}

\begin{proof}
    Let $h: 2^{\widehat{N}} \rightarrow \mathbb{R}$ be the construction in (\ref{E12}). The M\"obius inversion $\mu_{h}(S)=\mu_g(S)$ for every $S \in 2^N$ and
    \[
    \begin{split}
        \mu_{h}(S\widehat{e}) &=\sum\limits_{X \subseteq S\cup \widehat{e}}(-1)^{|S|+1-|X|}h(X)\\
        &=\sum\limits_{X \subseteq S}(-1)^{|S|+1-|X|}g(X)+\sum\limits_{X \subseteq S}(-1)^{|S|+1-(|X|+1)}\left(f(N)-f(N-X)\right)\\
        &=(\mu_{f^{\natural}}-\mu_g)(S).
    \end{split}
    \]
    By Theorem \ref{thm12} and \ref{thm13}, we obtain the decomposition (\ref{E11}). 
    
\end{proof}

Now we apply equation \ref{E11} to g-polypositroids. In particular, we are more interested in those g-polypositroids whose coefficients in the simplex decomposition are all nonnegative. In \cite{LP}, the authors provided some special polypositroids with all-nonnegative coefficients, such as cyclohedra, especially associahedra. However, in fact, not all g-polypositroids have nonnegative decomposition coefficients. A typical example is the hypersimplex $\Delta_{k,n}:=\{x \in [0,1]^n\;|\;x([n])=k\}$.
For example, taking $k=2$ and $n=4$, we have
\[
\Delta_{2,4}=\Delta_{\{1,2,3\}}+\Delta_{\{1,2,4\}}+\Delta_{\{1,3,4\}}+\Delta_{\{2,3,4\}}-2\Delta_{[4]}
\]
which has a negative term

A function \( h: \mathcal{I}(n) \to \mathbb{R} \) (\( \mathbb{Z} \)) has the following extension \( \tilde{h}: 2^{[n]} \to \mathbb{R} \) (\( \mathbb{Z} \)), defined by 
\begin{equation}
	\label{E14}
	\tilde{h}(A) = h(I_1) + \cdots + h(I_k), 
\end{equation}
where \( A = I_1 \cup \cdots \cup I_k \) is a covering of \( A \) by disjoint intervals, such that \( I_j \cup I_{j+1} \) fails to be an interval.
It is easy to check that if \( h \) is submodular/supermodular, then \( \tilde{h} \) is also submodular/supermodular. In this case, we call $\tilde{h}$ a \textbf{function on intervals}. Clearly, we have
\[
\Delta(\tilde{h})=\Delta(h):=\{x \in \mathbb{R}^n \;|\; x(I) \geq h(I) \; \text{and}\; x([n])=h([n])\; \text{for}\;I \in \mathcal{I}\backslash [n]\}.
\]




Given $S \in 2^{[n]}$, let $m(S)$ and $M(S)$ denote the minimum and maximum elements of $S$, respectively.

Now suppose that $h$ is a supermodular function on intervals. Let $I \in \mathcal{I}(n)$, and $I_m=I\backslash\{m(I)\}$  $I_M=I\backslash \{M(I)\}$. By supermodularity of $h$,
 \begin{equation}
    \label{E15}
    h(I)+h(I_m \cap I_M)-h(I_m)-h(I_M) \geq 0.
 \end{equation}
 Take (\ref{E16}) into (\ref{E15}), this is equivalent to  
\begin{equation}
    \label{E17}
    \sum\limits_{S \subseteq I}\mu_{\widetilde{h}}(S)+\sum\limits_{T \subseteq I_m \cap I_M}\mu_{\widetilde{h}}(T)-\sum\limits_{R \subseteq I_m}\mu_{\widetilde{h}}(R)-\sum\limits_{L \subseteq I_M}\mu_{\widetilde{h}}(L)=\mu_{\widetilde{h}}(I) \geq 0.
\end{equation}
When $A=I_1 \cup \cdots I_k$ $(k \geq 2)$ is the union of disjoint intervals such that $I_j \cup I_{j+1}$ fails to be an interval,
\begin{equation}
    \label{E18}
    \begin{split}
        \mu_{\widetilde{h}}(A)&=\sum\limits_{S \subseteq A}(-1)^{|A|-|S|}\widetilde{h}(S)\\
        &=\sum\limits_{S_1 \subseteq I_1}(-1)^{|I_1|-|S_1|} \left(\cdots \sum\limits_{S_k \subseteq I_k}(-1)^{|I_k|-|S_k|}\left(\widetilde{h}(S_1)+\cdots +\widetilde{h}(S_k)\right)\right)\\
        &=\sum\limits_{S_1 \subseteq I_1}(-1)^{|I_1|-|S_1|}\left(\cdots \sum\limits_{S_{k-1} \subseteq I_{k-1}}\mu_{\widetilde{h}}(I_k)\right)=0
    \end{split}
\end{equation}

\vspace{0.3cm}
%
From Definition \ref{def3}, we can see that when $h$ is a function on intervals, $\Delta(h)$ is a polypositroid. From this we immediately obtain the following proposition.

\begin{prop}
	The polytope
	\begin{equation}
		Q=\sum\limits_{I \in \mathcal{I}(n)}y_I\Delta_I+\sum\limits_{k=1}^n\overline{y}_k\overline{\Delta}_{[k]}
	\end{equation}
	is a g-polypositroid when $y_I \geq 0,\; \overline{y}_k \geq 0$ for every $I \in \mathcal{I}(n)$ and $k \in [n]$.
\end{prop}

\begin{proof}
	First we assume $\overline{y}_k=0$ for all $k \in [n]$, i.e., $	Q=\sum\limits_{I \in \mathcal{I}(n)}y_I\Delta_I$. For any $S \in 2^{[n]}$, we define $h(S):=\sum\limits_{I \subseteq S} y_I$. Since $y_I$ is nonnegative for any $I \in \mathcal{I}(n)$, it can be verified that h is a supermodular function on intervals. Thus $Q$ is a polypositroid. When not all $\overline{y}_k$
	are zero, $Q$ is the projection of a polypositroid constructed as above. Then by (3) of Proposition \ref{prop13}, $Q$ is a g-polypositroid. 
\end{proof}
%
%
\vspace{0.3cm}
More generally, for a g-polypositroid that admits a simplex decomposition with nonnegative coefficients, i.e.,
\begin{equation}
	\label{E19}
	Q=\sum\limits_{S \in 2^{[n]}}y_S\Delta_S+\sum\limits_{T \in 2^{[n]}}\overline{y}_T\overline{\Delta}_{T},\; \text{where} \; y_S,\; \overline{y}_T \geq 0.
\end{equation}

Naturally we can propose a question.

\begin{que}
	In \ref{E19}, under what conditions on $S$, $T$ and the preceding coefficients, is $Q$ a g-polypositroid?
\end{que}

After characterizing this problem, we can determine whether the sum of any two g-polypositroids of this form is again a g-polypositroid.

\vspace{0.5cm}

\subsection{Generalized positroids} \quad

\begin{defi}
    When a polyhedron is both a generalized polypositroid and a generalized matroid, we call it a \textbf{generalized positroid} (abbreviated as \textbf{g-positroid}).
\end{defi}

\vspace{0.2cm}
\begin{prop}
\label{prop1}
    Let $Q=Q(f,g)$ be a full-dimensional generalized matroid in $\mathbb{R}^n$, then there exists an edge of $Q$ which is a parallel translate of $e_i$ for any $i \in [n]$.
\end{prop}

\begin{proof}
    Since $Q$ is full-dimensioanl, we can choose an integer $k$ such that  $g([n])\leq k <k+1 \leq f([n])$. Let $\pi: \mathbb{R}^n \rightarrow \mathbb{R}^{n+1}$ be a lift mapping that takes $x=(x_1,\cdots,x_n)$ to $\pi(x)=(x_1,\cdots,x_n,k+1-\sum\limits_{i=1}^nx_i)$. We take the truncated generalized matroid $Q[k,k+1]:=Q \cap\{k \leq \sum\limits_{i=1}^nx_i \leq k+1\}$. Then $\pi(Q[k,k+1])$ is a full-dimensional matroid polytope in $\mathbb{R}^{n+1}$. So the corresponding matroid is connected in $[n+1]$ and there exists an edge of $\pi(Q[k,k+1])$ which is a parallel translate of $e_i-e_{n+1}$ (see \cite{BGW}). Then by projecting in the direction of $e_{n+1}$, we have found an edge as a parallel translate of $e_i$ for $Q[k,k+1]$, thus also for $Q$.
\end{proof}

\vspace{0.2cm}
\begin{prop}
\label{prop2}
    If $Q=Q(f,g)$ is an g-polypositroid, then Q does not contain two-dimensional face with vertices of the form $\{e_{Xij},e_{Xjk},e_{Xk},e_{Xi}\}$ with $\{i,j,k\} \subseteq [n]\backslash X$ and $i<j<k$ or the form of $\{e_{Xij},e_{Xjk},e_{Xkl},e_{Xli}\}$ with $\{i,j,k,l\} \subseteq [n]\backslash X$ and $i<j<k<l$.
\end{prop}

\begin{proof}
    We first show that if $e_{Xij},e_{Xjk},e_{Xk},e_{Xi} \in Q$, then $e_{Xik},e_{Xj} \in Q$. Among them, $e_{Xj} \in Q$ is euqivalent to that $g(I) \leq |Xj\cap I| \leq f(I)$ for any interval $I \in \mathcal{I}$. Since $e_{Xij} \in Q$, we have $|Xj \cap I| \leq |Xij\cap I| \leq f(I)$ which gives the right hand side of the above inequality.  On the other side, suppose that there exists an $I_0 \in \mathcal{I}$ such taht $|Xj\cap I_0|<g(I_0)$. Since $g(I_0) \leq |Xij \cap I_0|$ and $g(I_0) \leq |Xjk \cap I_0|$, we obtain that $i,k \in I_0$.
    Thus $|Xk\cap I_0|=|Xi \cap I_0|=|Xj \cap I_0| < g(I_0)$, a contradiction to the assumption that $e_{Xi}, e_{Xk} \in Q$. Therefore $e_{Xj} \in Q$ and similarly $e_{Xik} \in Q$. The six  points $\{e_{Xij},e_{Xjk},e_{Xk},e_{Xi},e_{Xik},e_{Xj}\}$ form an octahedron, so $Q$ does not contain the two dimensional face $\{e_{Xij},e_{Xjk},e_{Xk},e_{Xi}\}$. By similar arguments, $Q$ also does not contain $\{e_{Xij},e_{Xjk},e_{Xkl},e_{Xli}\}$ as a two-dimensional face.
\end{proof}

\begin{prop}
\label{prop3}
    Let $Q$ be a full-dimensional generalized matroid in $\mathbb{R}^n$ with $n\geq 3$. If $Q$ is not a generalized positroid, then there exists a two-dimensional face with vertice of the form $\{e_{Xij},e_{Xjk},e_{Xk},e_{Xi}\}$ with $\{i,j,k\} \subseteq [n]\backslash X$ and $i<j<k$.
\end{prop}

\begin{proof}
    Since $Q$ is not a generalized  positroid, there exists a facet $Q_1=\{x\;|\;\sum\limits_{i \in L}x_i=M\}\cap Q$ where $L$ is not an interval. From Chapter 3 of \cite{FT}, we know every face of a g-polymatroid is a g-polymatroid. Thus we have a strong pair $(f_1,g_1)$ such that $Q_1=Q_1(f_1,g_1)$ and $\sum\limits_{i \in L}x_i=f_1(L)=g_1(L)$, that is, $L$ is a sum-set for strong pair $(f_1,g_1)$. On the other hand, since $dim(Q_1)=n-1$, we get that $L$ is the unique inclusion-minimal sum-set by Corollary \ref{cor1}. Thus by Theorem \ref{thm1}, $Q_1=P_1\times Q_0$ where $P_1$ is a full-dimensional matroid polytope in $\mathbb{R}^L$ and $Q_0$ is a full-dimensional generalized matroid in $\mathbb{R}^{[n]\backslash L}$. Because $L$ is not an interval, we can find $a<b<c$ such that $a,c \in L$ and $b \in [n]\backslash L$.
    Then by the fact that the  corresponding matroid $M(P_1)$ is connected and Proposition \ref{prop1}, there exists an edge with vertices $\{e_{Aa}.e_{Ac}\}$ in $P_1$ and an edge with vertices $\{e_{B},e_{Bb}\}$ in $Q_0$. So the two-dimensional face with vertices $\{e_{Aa}.e_{Ac}\} \times \{e_{B},e_{Bb}\}=\{e_{ABab},e_{ABbc}.e_{ABa},e_{ABc}\}$ is a face of $Q_1$
\end{proof}
\vspace{0.2cm}

A \textbf{non-crossing partition} of $N$ is a partition $N=\bigsqcup\limits_{i}N_i$ such that for any two subsets $N_{i_1}$ and $N_{i_2}$, there do not exist indices $1 \leq a<b<c<d \leq n$ with $a,c \in N_{i_1}$ and $b,d \in N_{i_2}$. A subset $S$ \textbf{surrounds} a subset $T$ ($S \cap T= \emptyset$) if there exists a non-trivial partition $S=S_1 \sqcup S_2$ such that $\max(S_1)<\min(T)$ and $\max(T)<\min(S_2)$.

\vspace{0.2cm}
\begin{thm}
\label{thm2}
    Let $N=[n]$ equipped with an order, and $Q$ is a g-matroid, then the following statements are equivalent,
    \begin{enumerate}
        \item $Q$ is a g-positroid;
        \item there exists a unique non-crossing partition $[n]=S_1 \sqcup \cdots \sqcup S_k \sqcup T$ with no $S_i$ surrounds $T$, and a decomposition
        \[
        Q=P_1 \times\cdots \times P_k \times Q_0
        \]
        such that each $P_i$ is a full-dimensional positroid polytope in $\mathbb{R}^{S_i}$ and $Q_0$ is a full-dimensional g- positroid in $\mathbb{R}^{T}$;
        \item $Q$ contains no two-dimensional face with vertices of the form $\{e_{Xij},e_{Xjk},e_{Xk},e_{Xi}\}$ with $\{i,j,k\} \subseteq [n]\backslash X$ and $i<j<k$ or the form $\{e_{Xij},e_{Xjk},e_{Xkl},e_{Xli}\}$ with $\{i,j,k,l\} \subseteq [n]\backslash X$ and $i<j<k<l$.
    \end{enumerate}
\end{thm}

\begin{proof}
    (1) $\Rightarrow$ (2) By Theorem \ref{thm1}, such a decomposition exists. So first we need to show that  $[n]=S_1 \sqcup \cdots \sqcup S_K \sqcup T$ is a noncrossing partition and no $S_i$ surrounds $T$. If not, there exsits $a<b<c<d$ such that $a,c \in S_i$ and $b,d \in S_j$ for some $i \neq j$. Since every matroid polytope $P_i$ is full-dimensional in $\mathbb{R}^{S_i}$ ($M(P_i)$ is a connected matroid with base set $S_i$), $P_i$ has an edge with vertices $\{e_{Aa},e_{Ac}\}$ while $P_j$ has an edge with vertices $\{e_{Bb},e_{Bd}\}$. Therefore there exists a two-dimensional face with vertices$\{e_{ABab},e_{ABbc},e_{ABcd},e_{ABda}\}$ in $P_i \times P_j$ (thus in $Q$), contradicts to Proposition \ref{prop2}. Or if there exists $S_i$ such that $S_i$ and $T$ are not non-crossing or $S_i$ surrounds $T$, then there exists $a<b<c$ such that $a,c \in S_i$ and $b \in T$. Similarly, there exists an edge with vertices $\{e_{Aa},e_{Ac}\}$ in $P_i$ and an edge with vertices $\{e_{B},e_{Bb}\}$ in $Q_0$ by Proposition \ref{prop1}. Therefore we can find a two-dimensional face in $P_i\times Q_0$(thus in Q) with vertices $\{e_{ABab},e_{ABbc}.e_{ABa},e_{ABc}\}$ which contradicts to Proposition \ref{prop2}. Next, we prove that every $P_i$ is a positroid polytope in $\mathbb{R}^{S_i}$ and $Q_0$ is a g- positroid in $\mathbb{R}^T$ where the order in sets $S_i$ and $T$ are inherited from $[n]$. Let $f|_{S_i}$ be the restriction of $f$ on $S_i$ ($f|_{T},g|_{S_i},g|_{T}$ are defined respectively). Since $S_i$ is a sum-set, it is not hard to prove that for any $X \subseteq [n]$
    \begin{align*}
        f(X) &=f|_{S_i}(X \cap S_i)+\cdots+f|_{S_k}(X \cap S_k)+f|_{T}(X \cap T)\\
         g(X) &=g|_{S_i}(X \cap S_i)+\cdots+g|_{S_k}(X \cap S_k)+g|_{T}(X \cap T).
    \end{align*}
    Since $Q$ is a g-polypositroid, its facets are entirely determined by the values of $f$ and $g$ on the intervals. Therefore the values of $f|_{S_i}$ ($f|_{T},g|_{S_i},g|_{T}$ respectively) are entirely determined by the intervals in $S_i$ ($T,S_i,T$ respectively). Then each $P_i=P_i(f|_{S_i},g|_{S_i})$ is a positroid polytope and $Q_0=Q_0(f|_{T},g|_{T})$ is a g-positroid.\\\\
    (2) $\Rightarrow$ (1) Suppose that $P_i=P_i(f_i,g_i)$ and $Q_0=Q_0(f_0,g_0)$. Now given any $X \subseteq [n]$ and the partition  $X=X_1 \sqcup \cdots \sqcup X_k \sqcup X_0$ such that $X_0 \subseteq T$ and $X_i \subseteq S_i$ for $i=1,\cdots,k$. Define $f(X):=\sum\limits_{i=0}^kf_i(X_i)$ and $g(X):=\sum\limits_{i=0}^kg_i(X_i)$, next we prove that $Q=Q(f,g)$ is a g-positroid. Here, we mainly exploit the property that the facet of the polytope product is formed by the products of the facet of some component polytope with the other polytopes. Without loss of generality, we assume that the facet of $Q$ is defined by one of the equations $x(I_i)=f_i(I_i)$ ($i=0,1,\cdots,k$), where $I_0$ is an interval in $T$ and $I_i$ $(i \geq 1)$ is an interval in $S_i$ . But $I_i$ might not be an interval in $[n]$, Suppose that $I_i$ is uniquely expressed as a pairwise disjoint union of intervals as follows:
    \[
    I_i=I_i^1 \sqcup I_i^2 \sqcup \cdots \sqcup I_i^l
    \]
    where $I_i^1 < I_i^2 < \cdots < I_i^l$. Let $C_i^{k}$ be the interval bridging $I_i^k$ and $I_i^{k+1}$, then we complete $I_i$ into an interval $C(I_i)$ in $[n]$, that is,
    \[
    C(I_i):=[ min\;I_i^1,\;max\;I_i^l ]=I_i^1 \sqcup C_i^1 \sqcup I_i^2 \sqcup \cdots C_i^{l-1}\sqcup I_i^l.
    \]
    Since $[n]=S_1 \sqcup \cdots \sqcup S_k \sqcup T$ is a non-crossing partition with no $S_i$ surrounds $T$, every $C_i^k$ is a disjoint union of some $S_j$'s $(j \neq i)$. Thus the equation $x(I_i)=f_i(I_i)$ is equivalent to 
    \[
    x(C(I_i))=f_i(I_i)+\sum\limits_{S_j \subseteq C(I_i) }f_j(S_j)
    \]
    This indicates that the support of the facet-defining equation for polyhedron $Q$ constitutes an interval. So $Q$ is a g-positroid.
    \\

    (1) $\Rightarrow$ (3) is directly by Proposition \ref{prop1}. Since (1) $\Leftrightarrow$ (2), we prove that (3) $\Rightarrow$ (2). Suppose that $Q$ contains no two-dimensional faces of this two forms, by the same arguments in the proof of (1) $\Rightarrow$ (2), $[n]=S_1 \sqcup \cdots \sqcup S_k \sqcup T$ is a non-crossing partition and no $S_i$ surrounds $T$. Together with Theorem 3.9 in \cite{LPW} and Proposition \ref{prop3}, every $P_i$ must be a positroid polytope and $Q_0$ must be a g-positroid.
    
\end{proof}

\vspace{0.2cm}

\begin{thm}
	\label{affine_DCTP}
    For any DCTP function $f$, the affine regions of its concave extension on the hypercube $[0,1]^n$ belong to $\mathcal{GPP}(n)$. Moreover, the fan $\mathcal{MV}$ is exactly the secondary fan of the g-positroid subdivisions of the hypercube $[0,1]^n$.
\end{thm}
\begin{proof}
    Since a DCTP function is an $M^{\natural}$-function, all affine regions of the concave extension of $f$ are g-matroids. Furthermore, since the affinity areas of $f$ contain no two-dimensional faces of the forms in (3) of Theorem \ref{thm2}, thus they are g-positroids. Conversely, suppose that $f$ is an $M^{\natural}$-concave function on the domain $\{0,1\}^n$. As in the proof of Proposition \ref{plucker_submodular}, we define $\widetilde{f}$ on $\binom{[2n]}{n}$ by $\widetilde{f}(X):=f(X \cap [n])$. Since $f$ is $M^{\natural}$-concave if and only if $\widetilde{f}$ is M-concave, by \cite{AS}, this is equivalent to $\widetilde{f}$ being a Pl\"ucker vector i.e. $f$ is submodular and among the three sums 
    \[
    f(Xij)+f(Xk),\quad\quad  f(Xik)+f(Xj),\quad\quad  f(Xjk)+f(Xi)
    \]
    where $i<j<k$, two are equal and greater than or equal to the third. Since $f$ induces a g-positroid subdivision, by statement (3) of Theorem \ref{thm2}, $f$ should satisfy equation \ref{trop_plucker} thus is a DCTP function. Moreover, the two relations of DCTP functions determine the 3-dimensional skeleton complex of the subdivision, so the fan $\mathcal{MV}$ is exactly the secondary fan of the g-positroid subdivision of $[0,1]^n$.
\end{proof}

\vspace{1cm}

\vspace{1cm}
\section{Further applications and discussions}
\subsection{Constructing gereralized positroids from generalized cyclic patterns}\quad

A domian $\mathcal{D} \subseteq 2^{[n]}$ is called \textbf{w-pure} (\textbf{s-pure}) if all the maximal-by-inclusion weakly (strong) separated sets in $\mathcal{D}$ have the same size. The authors in \cite{DKK2} constructed a large amount of w-pure domains using a new geometric model called the \textbf{combined tiling}.

More precisely, we fix $n$ vectors $\xi_1,\xi_2,\cdots,\xi_n$ of unit Euclidean length in the upper half plane $\mathbb{R} \times \mathbb{R}_{>0}$, which

\begin{enumerate}
    \item  follow in this order clockwise around the origin $(0,0)$;
    \item are $\mathbb{Z}$-independent, i.e., all integer combinations of these vectors are different.
\end{enumerate}
The zonogon
\[
Z_n:=\{\lambda_1\xi_1+\lambda_2\xi_2+\cdots+\lambda_n\xi_n\;|\;0\leq \lambda_i \leq 1, \forall i \in [n]\}.
\]
is a $2n$-gon with vertices $(0,0)$ and $\xi_1,\;(\xi_1+\xi_2),\;\cdots,\;(\xi_1+\xi_2+\cdots+\xi_{n}),\;(\xi_2+\xi_3+\cdots+\xi_n),\;\cdots,\;\xi_n $ in the clockwise order. For any other $X \in 2^{[n]}$, the point $\xi_X:=\sum\limits_{i \in X}\xi_i$ is in the interior of $Z_n$ while it is not vertex of $Z_n$. Moreover, if $X \neq Y$, the points $\xi_X$ and $\xi_Y$ do not overlap.

Let $\mathcal{S}=(S_1,S_2,\cdots,S_r=S_0)$ be a sequence of  subsets in $2^{[n]}$ such that either $|S_{p-1}\triangle S_p|=1$ or $|S_{p-1}\triangle S_p|=2$ and $|S_{p-1}|=|S_p|$ for $p=1,2,\cdots,r$. The pair $\{S_{p-1},S_p\}$ in the former (latter) case is called a \textbf{1-distance pair} (resp. \textbf{2-distance pair}). Additionally, $\mathcal{S}$ is a weakly separated set. Such a sequence $\mathcal{S}$ is called a \textbf{generalized cyclic pattern} in \cite{DKK2}.

By connecting the adjacent vertices $\xi_{S_{p-1}}$ and $\xi_{S_p}$ in $\mathcal{S}$ with straight line segments, we get a closed piecewise-linear curve $\xi_{\mathcal{S}}$. In \cite{DKK2}, the authors provide a necessary and sufficient condition of $\xi_{\mathcal{S}}$ being non-self-intersecting.
Once $\xi_{\mathcal{S}}$ is non-self-intersecting, the zonogon $Z_n$ is divided into the region (include the curve $\xi_{\mathcal{S}}$, i.e.,  a closed set) inside $\xi_{\mathcal{S}}$ and the region outside $\xi_{\mathcal{S}}$, denoted by $\mathcal{R}_{\mathcal{S}}^{in}$ and $\mathcal{R}_{\mathcal{S}}^{out}$ respectively. Then let
\[
\mathcal{D}_{\mathcal{S}}^{in}:=\{X \in 2^{[n]}\;|\;\xi_{X} \in \mathcal{R}_{\mathcal{S}}^{in}\}\;\;\text{and}\;\;\mathcal{D}_{\mathcal{S}}^{out}:=\{X \in 2^{[n]}\;|\;\xi_{X} \in \mathcal{R}_{\mathcal{S}}^{out}\}
\]

\begin{thm}[\cite{DKK2}]
    Both $\mathcal{D}_{\mathcal{S}}^{in}$ and $\mathcal{D}_{\mathcal{S}}^{out}$ are w-pure domain, and for any maximal weakly separated sets $\mathcal{C}_1 \subseteq \mathcal{D}_{\mathcal{S}}^{in}$ and $\mathcal{C}_2 \subseteq \mathcal{D}_{\mathcal{S}}^{out}$, the union $\mathcal{C}_1 \cup \mathcal{C}_2$ is maximal weakly separated in the domain $2^{[n]}$.
\end{thm}
In particular, the Grassmann necklace mentioned in \cite{OPS} is a special class of generalized cyclic patterns. A \textbf{Grassmann necklace} is a sequence of $k$-elements $\mathcal{S}=(S_{n+1}=S_1,S_2,\cdots,S_n)$ such that either $S_{i}=S_{i+1}$ or $S_i \backslash S_{i+1}=\{i\}$ for $i \in [n]$. In other words, $\mathcal{S}$ is a generalized cyclic pattern $(S_{i_{r+1}}=S_{i_1}, S_{i_2}, \cdots, S_{i_r})$ such that $S_{i_j} \backslash S_{i_{j+1}}=\{i_j\}$ for $j \in [r]$ where $r \leq n$ and $ 1 \leq i_1 <i_2< \cdots <i_r \leq n$. Oh proved a one-to-one correspondence between Grassmann necklaces and positroids in \cite{O}.

\begin{thm}[\cite{O})]
    Let $\Pi_i=[i,i+1,\cdots,i-1]$ be the cones defined in Definition \ref{def1} for every $i \in [n]$ and $\mathcal{S}=(S_1,S_2,\cdots,S_n)$ be a Grassmann necklace. Then
    \[
    \mathcal{M}(\mathcal{S})=\bigcap\limits_{i=1}^n\left(\left(e_{S_i}-\Pi_i\right)\cap \Delta_{k,n}\right)
    \]
    is a positroid polytope. Conversely, for any positroid $\mathcal{M}$, take $S_i$ to be the minimal element in $\mathcal{M}$ under the cyclic shifted order $<_i$ (i.e., $i<_ii+1<_i\cdots<_ii-1$), then $\mathcal{S}$ is a Grassmann necklace.
\end{thm}

Here we utilize similar methods to generate a generalized positroid from a generalized cyclic pattern.

\begin{prop}
    \label{prop10}
    Suppose that $\widetilde{\mathcal{S}}=\{\widetilde{S}_1,\widetilde{S}_2,\cdots,\widetilde{S}_{2n}\}$ is a connected Grassmann necklace in the domain $\binom{[2n]}{n}$, and $\mathcal{S}=\{S_1,S_2,\cdots,S_r\}$ is a (non-self-intersecting) generalized cyclic pattern obtained from $\widetilde{\mathcal{S}}$ by taking the intersection $\widetilde{S}_i \cap [n]$ for $i \in [2n]$. Then,
    \[
    \mathcal{GM}(\mathcal{S}):=\text{Conv}\{e_Y\;|\;Y \in \mathcal{D}_{\mathcal{S}}^{in}\}
    \]
    is a generalized positroid where \text{Conv}$\{\cdot\}$ denotes taking the convex hull.
\end{prop}

\begin{proof}
First, let us define a map $\gamma$ between two collections of points, 
\[
\gamma: \xi_X \in Z_{2n} \rightarrow \xi_{X \cap [n]} \in Z_n.
\]
Let $\phi:[0,1] \rightarrow \mathbb{R}^2$ be a parameterization of necklace $\widetilde{\mathcal{S}}$ such that $\phi(k/2n)=\xi_{\widetilde{S}_k}$, and for $t \in [k/2n,(k+1)/2n]$, define $\phi(t)=(2nt-k)\xi_{\widetilde{S}_k}+(k+1-2nt)\xi_{\widetilde{S}_{k+1}}$. So we can extend the map $\gamma$ between points of patterns $\widetilde{\mathcal{S}}$ and $\mathcal{S}$ to the curves by defining $\gamma(\phi(t))=(2nt-k)\gamma(\xi_{\widetilde{S}_k})+(k+1-2nt)\gamma(\xi_{\widetilde{S}_{k+1}})$. Furthermore, $\gamma$ can be extended to a continuous surjective map between the regions $\mathcal{R}_{\widetilde{\mathcal{S}}}^{in}$ and $\mathcal{R}_{\mathcal{S}}^{in}$ (as left of Figure \ref{fig29}). For example, we can choose some combined tiling $K$ on $\mathcal{R}_{\widetilde{\mathcal{S}}}^{in}$, and refine $K$ into a triangulation. Thus every point $X \in \mathcal{R}_{\widetilde{\mathcal{S}}}^{in}$ lies in some triangle $\{\xi_{X_1},\xi_{X_2},\xi_{X_3}\}$, and has a unique expression $X=\lambda_1\xi_{X_1}+\lambda_2\xi_{X_2}+\lambda_3\xi_{X_3}$ where $\lambda_1+\lambda_2+\lambda_3=1$. Therefore, we can define the image $\gamma(X):=\lambda_1\gamma(\xi_{X_1})+\lambda_2\gamma(\xi_{X_2})+\lambda_3\gamma(\xi_{X_3})$. 

Since $\widetilde{\mathcal{S}}$ is a Grassmann necklace, by Proposition 9.10 of \cite{OPS}, $X\in \mathcal{D}_{\widetilde{\mathcal{S}}}^{in}$  if and only if $e_X$ belongs to the positroid $\mathcal{M}(\widetilde{\mathcal{S}})$. Let $\pi:\mathbb{R}^{2n} \rightarrow \mathbb{R}^n$ be the projection to the first n coordinates. So $\pi(\mathcal{M}(\widetilde{S}))$ is a generalized positroid and for every $e_Y \in \pi(\mathcal{M}(\widetilde{S}))$, there exists some $X \in \mathcal{D}_{\widetilde{\mathcal{S}}}^{in}$ such that $\gamma(\xi_X)=\xi_Y$, so $\pi(\mathcal{M}(\widetilde{S})) \subseteq \mathcal{GM}(\mathcal{S})$. Conversely, for any $e_Y \in \mathcal{GM}(\mathcal{S})$, since $\gamma$ is surjective, there exists $X=\lambda_1\xi_{X_1}+\lambda_2\xi_{X_2}+\lambda_3\xi_{X_3} \in \mathcal{R}_{\widetilde{\mathcal{S}}}^{in}$ (a triangle in the triangulation of $K$ as above) such that $\gamma(X)=\xi_Y$. Since $e_{X_1}$, $e_{X_2}$ and $e_{X_3}$ are in positroid $\mathcal{M}(\widetilde{S})$, we get $e_Y=\pi(\lambda_1e_{X_1}+\lambda_2e_{X_2}+\lambda_3e_{X_3})$ is in $\pi(\mathcal{M}(\widetilde{\mathcal{S}}))$. By Proposition (3) of \ref{prop13}, this proves that $\mathcal{GM}(\mathcal{S})=\pi(\mathcal{M}(\widetilde{\mathcal{S}}))$ is a generalized positroid.
\end{proof}

\vspace{0.3cm}
In particular, we consider the generalization (see \cite{DKK2}) of the $\omega$-chamber introduced in the work of Leclerc and Zelevinsky \cite{LZ}. Let $\omega$ be a permutation on $[n]$ and let Inv($\omega$) denotes the set of inversions of $\omega$ (the pair $(i,j)$ such that $i<j$ and $\omega(i) > \omega(j)$). The collection of sets
\[
\mathcal{D}(\omega):=\{X \in 2^{[n]}\;|\;\text{for}\;i<j,\;\omega(i)<\omega(j)\;\text{and}\; j \in X,\;\text{then}\;i \in X\}
\]
is called the $\omega$-\textbf{chamber}. Additionally, if there are two permutations $\omega'$  and $\omega$ with Inv($\omega'$) $\subseteq$ Inv($\omega$), $\mathcal{D}(\omega',\omega)$ is defined to be the sets  $X \in \mathcal{D}(\omega)$ with additional condition: for $i<j$, $\omega'(i) >\omega'(j)$ and $i \in X$, then $j \in X$.

\begin{ex}
    Let $\omega_0$ be the longest permutation (where $\omega_0:i \rightarrow n-i+1$). For any permutation $\omega$ on $[n]$, the polytope
    \[
    \mathcal{P}(\omega):=Conv\{e_Y\:|\;Y \in \mathcal{D}(\omega,\omega_0)\}
    \]
    is a generalized positroid.
\end{ex}

\begin{proof}
    As shown in the right of Figure \ref{fig29}, $\mathcal{D}(\omega, \omega_0)=\mathcal{D}_{\mathcal{S}}^{in}$ where the generalized cyclic pattern $\mathcal{S}$ is given by $S_1=[n],\cdots ,S_k=[k,n],\cdots,S_n=\{n\},S_{n+1}=\emptyset,S_{n+2}=\omega^{-1}([1])\cdots,S_{n+k}=\omega^{-1}([k-1]),\cdots,S_{2n}=\omega^{-1}([n-1])$ (see \cite{DKK2}). Since $\widetilde{\mathcal{S}}=(\mathbf{pad}(S_1),\cdots,\mathbf{pad}(S_{2n}))$ is a Grassmann necklace in the domain $\binom{[2n]}{n}$, by Proposition \ref{prop10}, $\mathcal{P}(\omega)$ is a generalized positroid.
\end{proof}

 \begin{figure}[H]
	\centering
	\includegraphics[width=16cm]{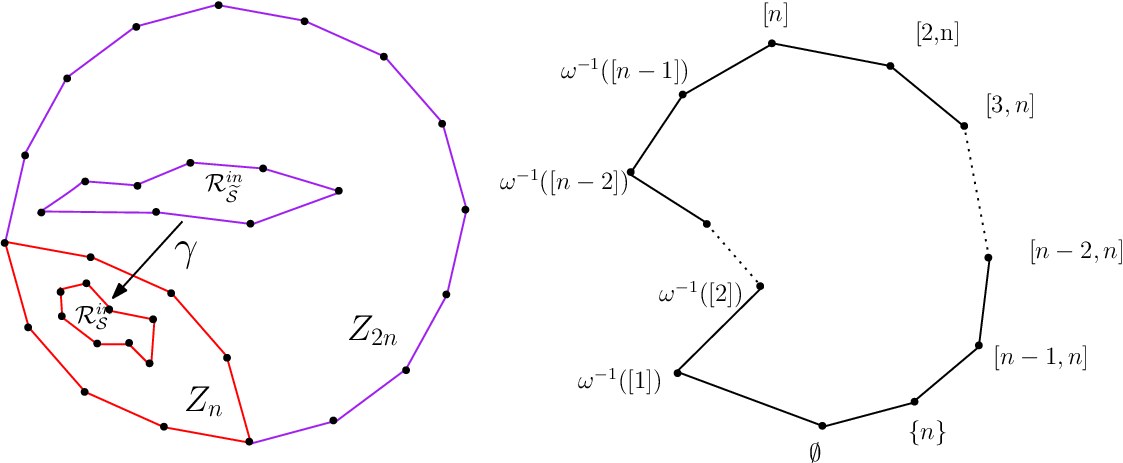}
	\caption{}
	\label{fig29}
\end{figure}
\vspace{0.5cm}
\subsection{Fan structure of stable MV polytopes}\quad\\
\indent As shown by Kamnitzer (see Section 6 in \cite{K2} ), the equivalence classes of MV polytopes under translation are called \textbf{stable MV polytopes}. Each MV polytope is determined by a unique DCTP function, and all DCTP functions form a fan structure under addition, thereby causing the stable MV polytopes to constitute a fan $\mathcal{MV}$ under Minkowski sums. Kamnitzer referred to the generators of all cones in the fan $\mathcal{MV}$ as \textbf{prime MV polytopes} and designated certain generators within the same cone as "\textbf{clusters}". 

Anderson computed all the prime MV polytopes for some lower-rank cases in \cite{A}. Later in \cite{AK}, Anderson and Kogan argued that these generators are closely related to the cluster algebra of Berenstein-Fomin-Zelevinsky \cite{BFZ}.

And here we formally state Kamnizer's question in \cite{K2} as follows,

\begin{prob}
    What are all the generators of the fan $\mathcal{MV}$ or prime MV polytopes? How to characterize the maximal cones of $\mathcal{MV}$?
\end{prob}

However, by relating the secondary fan of generalized positroid subdivision of $[0,1]^n$ to the fan of DCTP functions, we obtained that all the Schubert matroid polytopes are prime MV polytopes.

\begin{prop}
    \label{10}
    For $I \in 2^{[n]} \backslash \{[1,j]\;|\;j \in [n]\}$, Schubert matroid polytopes $Q(\theta_I)$ are prime MV polytopes. Moreover when $I \neq J$, $Q(\theta_I)$ and $Q(\theta_J)$ represent different stable MV polytopes.
\end{prop}
\begin{proof}
    From Theorem \ref{thm7}, we know that $\theta_I$ induce a muti-split of $[0,1]^n$ which is a coarsest subdivision by Theorem \ref{thm6}. Thus $\theta_I$ corresponds to a generator of $\mathcal{MV}$, i.e. $Q(\theta_I)$ is a prime MV polytope. Since two matroid polytopes are move equivalent if and only if they can be obtained from each other by replacing loops and coloops. So for $I \neq J$, $Q(\theta_I)$ and $Q(\theta_J)$ are not move equivalent.
\end{proof}

\vspace{0.2cm}
In Theorem \ref{minkow_schubert}, we have shown that, for any weakly separated set, the
Minkowski sums of Schubert matroid polytopes, labeled by sets of  a weakly separated set, form a subset of MV polytopes stable under summation. 

Moreover,  for a maximal weakly separates set $\mathcal{C} \subseteq 2^{[n]}$, the  MV polytope $Q_{\mathcal{C}}:=\sum\limits_{I \in \mathcal{C}\backslash \{[1,j]\;|\;j \in [n]\}}Q(\theta_I)$ corresponds to a maximal cone of $\mathcal{MV}$.
This follows from 

\begin{thm}
    \label{thm10}
    Suppose that $\mathcal{C} \subseteq 2^{[n]}$ is a maximal weakly separated set, let $\theta=\sum\limits_{I \in \mathcal{C}\backslash \{[1,j]\;|\;j \in [n]\}} \theta_I$. Then $\theta(Ta)+\theta(Tb) > \theta(T)+\theta(Tab)$ and $\theta(Tab)+\theta(Tc) \neq\theta(Tbc)+\theta(Ta)$ for every $a<b<c$ and $T \subseteq [n]\backslash \{a,b,c\}$.
\end{thm}

To prove this theorem, we need to use a result from our earlier work.

\begin{thm}[\cite{KLZ}]
    \label{thm11}
    Suppose that $\mathcal{C} \subseteq \binom{[n]}{k}$ is a maximal weakly separated set and $\mathcal{I} \subseteq \binom{[n]}{k}$ is the set of cyclic k-intervals. Let $\theta:=\sum\limits_{I \in \mathcal{C}\backslash\mathcal{I}}\theta_I$, then $\theta(Tab)+\theta(Tcd) \neq \theta(Tad)+\theta(Tbc)$ for every $a <b<c<d$ and $T \in \binom{[n]\backslash\{a,b,c,d\}}{k-2}$,i.e., the positroid subdivision of $\Delta_{k,n}$ induced by $\theta$ is the finest.
\end{thm}

\vspace{0.2cm}
For any collection of sets $\mathcal{C} \subseteq 2^{[n]}$, we can define $\mathbf{pad}(\mathcal{C}):=\{\mathbf{pad}(I)\;|\;I \in \mathcal{C}\}$. Also, we can define a shift of this pad construction. Given any $I \subseteq [n+1,2n]$, we define $\widetilde{\mathbf{pad}}(I):=[|I|+1,n] \sqcup I$ (when $I=[n+1,2n]$, $\widetilde{\mathbf{pad}}(I)=I$). So similarly if $\mathcal{D} \subseteq 2^{[n+1,2n]}$ is a collection of subsets of $[n+1,2n]$, then we let
\[
\widetilde{\mathbf{pad}}(\mathcal{D}):=\{\widetilde{\mathbf{pad}}(I)\;|\;I \in \mathcal{D},\;\mathcal{D} \subseteq 2^{[n+1,2n]}\}.
\]

Oh et al. \cite{OPS} have a lemma that also holds for this shifted pad construction with a similar proof.

\begin{lem}[\cite{OPS}]
    \label{lem6}
    For $I,J \in 2^{[n]}$ (resp. $I,J \in 2^{[n+1,2n]}$), they are weakly separated if and only if $\mathbf{pad}(I)$ and $\mathbf{pad}(J)$ (resp. $\widetilde{\mathbf{pad}}(I)$ and $\widetilde{\mathbf{pad}}(J)$) are weakly separated.
\end{lem}

\vspace{0.2cm}
\begin{lem}
    \label{lem7}
    If $\mathcal{C} \subseteq 2^{[n]}$ and $\mathcal{D} \subseteq 2^{[n+1,2n]}$ are both maximal weakly separated, then $\mathbf{pad}(\mathcal{C}) \cup \widetilde{\mathbf{pad}}(\mathcal{D})$ is maximal weakly separated over $\binom{[2n]}{n}$.
\end{lem}

\begin{proof}
    By Lemma \ref{lem6}, we know both $\mathbf{pad}(\mathcal{C})$ and $\widetilde{\mathbf{pad}}(\mathcal{D})$ are weakly separated set. Below we show that any two elements in $\mathbf{pad}(\mathcal{C})$ and $\widetilde{\mathbf{pad}}(\mathcal{D})$ are weakly separated to each other. Suppose that $\mathbf{pad}(I)=I \sqcup [n+|I|+1,2n] \in \mathbf{pad}(\mathcal{C})$ and $\widetilde{\mathbf{pad}}(J)=[|J|+1,n] \sqcup J \in \widetilde{\mathbf{pad}}(\mathcal{D})$, then we have 
    \[
    I\backslash [|J|+1,n] <[|J|+1,n]\backslash I <J \backslash [n+|I|+1,2n] < [n+|I|+1,2n] \backslash J.
    \]
    Together with the relations $\mathbf{pad}(I)\backslash \widetilde{\mathbf{pad}}(J)= (I\backslash [|J|+1,n]) \sqcup ([n+|I|+1,2n] \backslash J)$ and $\widetilde{\mathbf{pad}}(J) \backslash \mathbf{pad}(I)=([|J|+1,n]\backslash I) \sqcup (J \backslash [n+|I|+1,2n])$, we obtain that $\mathbf{pad}(I)$ and $\widetilde{\mathbf{pad}}(J)$ are weakly separated.

    The maximality follows from the fact (see \cite{DKK}) that the maximal weakly separeted sets in the domain  $2^{[n]}$ (resp. in the domain $\binom{[n]}{k}$) are of the same size $\binom{n+1}{2}+1$ (resp. size $k(n-k)+1$). Since $\mathbf{pad}(\mathcal{C}) \cap \widetilde{\mathbf{pad}}(\mathcal{D})=[1,n] \cup [n+1,2n] \cup \{[n-l+1,n] \sqcup [n+l+1,2n]\;|\;l=1,2,\cdots,n-1\}$, we obtain that $|\mathbf{pad}(\mathcal{C}) \cap \widetilde{\mathbf{pad}}(\mathcal{D})|=n+1$. Therefore,
    \[
    \begin{split}
        |\mathbf{pad}(\mathcal{C}) \cup \widetilde{\mathbf{pad}}(\mathcal{D})|&=|\mathbf{pad}(\mathcal{C})|+| \widetilde{\mathbf{pad}}(\mathcal{D})|-|\mathbf{pad}(\mathcal{C}) \cap \widetilde{\mathbf{pad}}(\mathcal{D})|\\
        &=2\left(\binom{n+1}{2}+1\right)-(n+1)\\
        &=n(2n-n)+1.
    \end{split}
    \]
    This proves that $\mathbf{pad}(\mathcal{C}) \cup \widetilde{\mathbf{pad}}(\mathcal{D})$ is maximal weakly separated over $\binom{[2n]}{n}$
\end{proof}

We are now ready to prove Theorem \ref{thm10}.
\begin{proof}[Proof of Theorem \ref{thm10}]
    We first prove that $\theta(Ta)+\theta(Tb)>\theta(T)+\theta(Tab)$. Suppose that $|Ta|=|Tb|=l \geq 1$ and let $L=[n-l+1,n]$. Thus $\Omega_L$ is a uniform matroid and the rank function $\theta_L(X)=|X|$ when $|X| \leq l$. Otherwise, $\theta_L(X)=l$. Thus,
    \[\theta_L(Ta)+\theta_L(Tb)=2l>\theta_L(T)+\theta_L(Tab)=l-1+l=2l-1.
     \]
     By the submodularity of all the other summands in $\theta$, we get the first inequality.

     To prove that $\theta(Tab)+\theta(Tc) \neq \theta(Tbc)+\theta(Ta)$, we begin by introducing some notations. Let $\mathcal{D}_0 \subseteq 2^{[n+1,2n]}$ be the collection of intervals in $[n+1,2n]$. Particularly, $\mathcal{D}_0$ is a maximal weakly separated set over $2^{[n+1,2n]}$. We use $\widetilde{\mathcal{I}}$ to denote the set of all the cyclic n-intervals in $[2n]$ and let $\mathcal{I}_1=\{[1,j]\;|\;j \in [n]\}$ and $\mathcal{I}_2=\{[n+1,j]\;|\;j \in [n+1,2n]\}$. Clearly, we have the following facts:
     \begin{itemize}
         \item By Lemma \ref{lem7}, $\widetilde{\mathcal{C}}:=\mathbf{pad}(\mathcal{C}) \cup \widetilde{\mathbf{pad}}(\mathcal{D}_0)$ is a maximal weakly separated set over $\binom{[2n]}{n}$.

         \item $\mathcal{I}_1 \subseteq \mathcal{C}$, $\mathcal{I}_2 \subseteq \mathcal{D}_0$ and $\widetilde{\mathcal{I}}=\mathbf{pad}(\mathcal{I}_1) \cup \widetilde{\mathbf{pad}}(\mathcal{I}_2)$, $\widetilde{\mathcal{C}}\backslash \widetilde{\mathcal{I}}=\mathbf{pad}(\mathcal{C}\backslash \mathcal{I}_1) \cup \widetilde{\mathbf{pad}}(\mathcal{D}_0 \backslash \mathcal{I}_2)$.
     \end{itemize}
     
     Then we define $\widetilde{\theta}:=\sum\limits_{J \in \widetilde{\mathcal{C}}\backslash\ \widetilde{\mathcal{I}}}\theta_J=\widetilde{\theta}_1+\widetilde{\theta}_2$ where
     $\widetilde{\theta}_1=\sum\limits_{I \in \mathcal{C}\backslash \mathcal{I}_1}\theta_{\mathbf{pad}(I)}$ and $\widetilde{\theta}_2=\sum\limits_{I \in \mathcal{D}_0\backslash \mathcal{I}_2}\theta_{\mathbf{pad}(I)}$.
     Then by Theorem \ref{thm11}, $\widetilde{\theta}(Tab)+\widetilde{\theta}(Tcd) \neq \widetilde{\theta}(Tbc)+\widetilde{\theta}(Tad)$ for every $a<b<c<d$ and $T \in \binom{[2n]\backslash\{a,b,c,d\}}{n-2}$.
     
     Furthermore, we  claim that $\widetilde{\theta}_2(Tab)+\widetilde{\theta}_2(Tcd)=\widetilde{\theta}_2(Tbc)+\widetilde{\theta}_2(Tad)$ for $a,b,c \in [n]$ and $d \in [n+1,2n]$. Through Lemma \ref{lem4}, we know that for any $I \in \mathcal{D}_0\backslash \mathcal{I}_2$ and $X \in \binom{[2n]}{n}$, 
     \begin{equation}
     \label{E9}
     \theta_{\widetilde{\mathbf{pad}}(I)}(X)=
     \begin{cases}
         |X \cap [n]|+|I| &\text{for}\;\;|X \cap [n+1,E(I)]| \geq |I|,\\
         |X \cap [1,E(I)]| &\text{for}\;\;|X \cap [n+1,E(I)]| \leq |I|.
     \end{cases}
     \end{equation}
     
     Since $a,b,c \in [n]$ and $d \in [n+1,2n]$, we have $|Tab \cap [n+1,E(I)]|=|Tbc \cap [n+1,E(I)]|$ and $|Tad \cap [n+1,E(I)]|=|Tbc \cap [n+1,E(I)]|$. Combining with equation \ref{E9}, we get that $\theta_{\widetilde{\mathbf{pad}}(I)}(Tab)+\theta_{\widetilde{\mathbf{pad}}(I)}(Tcd)=\theta_{\widetilde{\mathbf{pad}}(I)}(Tbc)+\theta_{\widetilde{\mathbf{pad}}(I)}(Tad)$ for $I \in \mathcal{D}_0 \backslash \mathcal{I}_1$. Then summing over all $I \in \mathcal{D}_0 \backslash \mathcal{I}_1$ proves the claim.

     Since $\widetilde{\theta}=\widetilde{\theta}_1+\widetilde{\theta}_2$, we immediately get that $\widetilde{\theta}_1(Tab)+\widetilde{\theta}_1(Tcd)  \neq \widetilde{\theta}_1(Tbc)+\widetilde{\theta}_1(Tad)$ for $a,b,c \in [n]$ and $d \in [n+1,2n]$. And Lemma \ref{lem5} tells us that $\widetilde{\theta}_1(X)=\theta(X \cap [n])+\sum\limits_{I \in \mathcal{C}\backslash\mathcal{I}_1}(n-|I|)$ for every $X \in \binom{[2n]}{n}$. So we obtain that 
     $\theta((T \cap [n])ab)+\theta((T \cap [n])c) \neq \theta((T \cap [n])bc)+\theta((T \cap [n])a)$. Since $T \cap [n]$ and $a,b,c \in [n]$ can be chosed arbitrarily, we complete the proof.
\end{proof}

Use the same notations as Theorem \ref{thm10}, we obtain
\begin{cor}
	\label{finest_g_positroid}
    The g-positroid subdivision of $[0,1]^n$ induced by $\theta$ is the finest. Moreover, the maximal weakly separated set $\mathcal{C}$ (in the domain $2^{[n]}$) corresponds to a maximal cone of $\mathcal{MV}$.
\end{cor}

\vspace{0.2cm}
\begin{rmk}
	Since the class of $g$-polypositroids is stable under central symmetry ($P\to -P$), each finest $g$-positroid subdivision $\theta$ of $[0,1]^{[n]}$ gives rise
	to a {\bf dual} $g$-positroid subdivision of $[0,1]^{[n]}$ by translating centrally symmetric subdivision of $[-1,0]^{n}$ induced by $-\theta$ by the unit vector $(1,\cdots, 1)$.
\end{rmk}

\vspace{0.2cm}
{\bf Acknowledgements.}\;
{\em  For this project, FL and LZ  are supported by the National Natural Science Foundation of China (No.12131015).}


\end{document}